\documentclass[letterpaper,12pt]{amsart}

\usepackage{tkz-euclide}
\usepackage{mathrsfs}

\usepackage[abbrev]{amsrefs}
\usepackage{amssymb,amsmath,amsthm}
\usepackage{color}
\usepackage[xcolor=red]{changes}

\usepackage{enumerate}
\usepackage{ulem,ifthen,xcolor,xkeyval,pdfcolmk}

\usepackage{graphicx}

\makeatletter
\@namedef{Changes@AuthorColor}{red}
\colorlet{Changes@Color}{red}
\makeatother

\definechangesauthor[color=blue]{nc}
\definechangesauthor[color=red]{zl}

\newtheorem{prop}{Proposition}[section]
\newtheorem{thm}{Theorem}[section]

\newtheorem{lem}{Lemma}[section]

\newtheorem{remark}{Remark}
\newtheorem{corl}{Corollary}[section]

\newtheorem{conj}{Conjecture}

\newcommand{\eps}{\varepsilon}

\def \<{\langle}
\def \>{\rangle}

\def \R{\mathbb R}
\def \C{\mathbb C}

\def \H{{\cal H}}

\def \H^0{{\cal H}^0 or}

\def \p{\partial}
\def \n{\nabla}
\def \beq{\begin{equation}}
\def \eeq{\end{equation}}

\def \n{\nabla}

\def \eref{\eqref}

\begin{document}



\title{$L^p$-spectral theory for the Laplacian on forms}

\author{Nelia Charalambous}
\address{Department of Mathematics and Statistics, University of Cyprus, Nicosia, 1678, Cyprus} \email[Nelia Charalambous]{nelia@ucy.ac.cy}

 \author{Zhiqin Lu} \address{Department of
Mathematics, University of California,
Irvine, Irvine, CA 92697, USA} \email[Zhiqin Lu]{zlu@uci.edu}

\thanks{The first author is partially supported by a University of Cyprus Internal Grant.
The second author is partially supported by the DMS-19-08513 }
 \date{\today}



\begin{abstract}
 In this article, we find sufficient conditions on an open Riemannian manifold so that a Weyl criterion holds for the $L^p$-spectrum of the Laplacian on $k$-forms, and also prove the decomposition of the $L^p$-spectrum depending on the order of the forms.   We then show that the resolvent set of an operator such as the Laplacian on $L^p$   lies outside a parabola whenever the volume of the manifold has an exponential volume growth rate, removing the requirement on the manifold to be of bounded geometry. We conclude by  providing a detailed description of the $L^p$  spectrum of the Laplacian on $k$-forms over hyperbolic space.
\end{abstract}

\maketitle

\section{Introduction} \label{S1}

In this article, we consider the Laplacian on differential $k$-forms over an open (complete noncompact) Riemannian manifold $(M,g)$. On smooth differential $k$-forms, the Laplacian is given by
\[
\Delta=d\delta +\delta d
\]
where $d$ is exterior differentiation and $\delta$ is the dual operator of $d$ with respect to the $L^2$-pairing. We are interested in the spectrum of the Laplacian acting on $L^p$-integrable $k$-forms $L^p(\Lambda^k(M))$ for all $1\leq p\leq \infty$. For $p\in[1,\infty]$ one can define the Laplacian operator on $L^p(\Lambda^k(M))$ via the heat semigroup. The $L^p$-spectrum  and essential spectrum for the Laplacian on $k$-forms, $\sigma(p,k)$ and $\sigma_\textup{ess}(p,k)$ respectively, are then defined in a similar way to the $L^2$-spectrum (see Section \ref{S0} for the details). Note that if the manifold is compact, the spectrum is trivially the same for all $p$, hence studying the $L^p$-spectrum only makes sense in the noncompact case.

 The $L^p$-spectrum is often more complex and hence more interesting in comparison to the $L^2$-spectrum. One of its interesting features  is the fact that it could actually vary with $p$ over noncompact manifolds, and this variation can be controlled by the geometric properties of the manifold.  The classic examples are hyperbolic space, Kleinian group and certain conformally compact manifolds, where the $L^p$-spectrum of the Laplacian on functions,  $\sigma_\textup{ess}(p,0)$, is
a parabolic region of the complex plane for all $p\neq 2$ and reduces to  an  unbounded  interval of the real line for $p=2$ ~\cites{ChRo,LR,DST}.

On the other hand, $\sigma (p,k)$ was proved to be independent of $p$ on manifolds with uniformly subexponential volume growth and whose Ricci curvature and curvature tensor on $k$-forms are bounded below (see ~\cites{HV1,HV2,sturm} for the case $k=0$ who consider both the Laplacian and Schr\"odinger operators, and ~\cite{CharJFA} for all other $k$).

Our first goal in this article is to study the Weyl criterion  for the $L^p$-spectrum. Surprisingly such a criterion is not known.
The classical Weyl criterion states that a point $\lambda$ will belong to the $L^2$-spectrum of the Laplacian if and only if we can find a sequence of approximate $L^2$-eigenforms corresponding to $\lambda$. This criterion does not generalize to the $L^p$-spectrum, in other words it is not clear that we can always find a sequence of approximate $L^p$ eigenforms whenever $\lambda$ is in the $L^p$-spectrum  (see Theorem \ref{Weyl})\footnote{In general very little is known for the spectrum of operators over a Banach space. For a discussion on how the residual and point spectrum are related see \cite{RSI}*{Proposition p. 194}.}.  Our first goal is to find sufficient conditions on the manifold so that the $L^p$-Weyl criterion holds for the Laplacian on forms. In Section~\ref{S0}, Theorem \ref{thm2}, using heat kernel bounds, we prove that for $p\geq 2$ the $L^p$-Weyl criterion holds whenever the Ricci curvature and the Weitzenb\"ock tensor are bounded below, and the volume of balls of radius $1$ is uniformly bounded below  by a positive constant. In the same section we also prove that a point in the $L^p$-spectrum on $k$-forms must come from either the $(k+1)$ or the $(k-1)$-form spectrum  (Proposition \ref{kspec}). We achieve these results by proving  precise upper Gaussian estimates for the heat kernel on forms and  its first order derivatives,  as well as pointwise and integral estimates for the resolvent operator and its derivatives. These are found in Section \ref{S2}.

Our next goal is to generalize the set of manifolds over which the $L^p$-spectrum of the Laplacian on $k$-forms is contained in a parabolic
region. First, we prove  the finite propagation speed of the wave kernel for the Laplacian on forms (Theorem \ref{thmFP}).
 Then in Section \ref{S4} we  show that the resolvent set of an operator such as the Laplacian on $L^p$ lies outside a parabola whenever the volume of the manifold has exponential volume growth rate (see Section \ref{S4} for the precise definition).
\begin{thm} \label{thm3}
We consider a complete manifold $M$ over which the Ricci curvature and the Weitzenb\"ock tensor on $k$-forms are bounded below.  Denote by $\gamma$  the exponential rate of volume growth of $M$ and $\lambda_1$   the infimum of the spectrum of the Laplacian on $L^2$ integrable $k$-forms.     Let $1\leq p\leq \infty$, and $z$ be a complex number such that
\[
|{\rm Im}(z)|>\gamma\left|\frac 1p-\frac 12\right|.
\]
 Then,  for any real number $c$ satisfying $c\leq\lambda_1$, the operator
$
(H-z^2)^{-1}
$
is  bounded on $L^p(\Lambda^k(M))$, where
\[
H=\Delta-c
\]
is a nonnegstive operator in $L^2(\Lambda^k(M))$.
\end{thm}
Our result generalizes that of Taylor \cite{tay} in that it removes the requirement that the manifold should be of bounded geometry. In particular, we no longer require that the injectivity radius of the manifold be bounded below. We achieve this through our precise estimates for the resolvent kernel and by using the resolvent to effectively replace the parametrix method in \cite{tay}  for the wave kernel.

 We visualize the result in Theorem~\ref{thm3} with the picture below, where the set outside the shaded parabolic region represents the set of complex numbers $w$ for which the resolvent $(H-w)^{-1}$ is bounded in $L^p$.  Note that a parabolic region is always contained in a cone with vertex on the real line. It is worth mentioning that in \cites{LiPe} it was proved that the $L^p$-spectrum is always contained in a cone which depends only on $p$ and in \cite{We} Weber provides an example of a surface where the $L^p$-spectrum is the maximal cone. However, upon careful examination of their results, we see that in the case $p=1$, the cone degenerates into a half-plane. Our results imply that under the assumption of a lower bound for the Ricci and Weizenb\"ock tensor, there exists a strict cone that contains the spectrum, confirming the following conjecture in this special case. \\

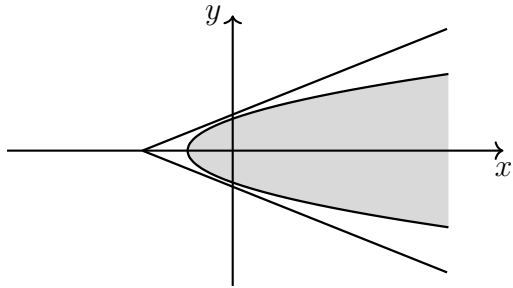
\begin{center}
\begin{tikzpicture}[scale=.6]

  \draw[black, line width = .3mm, fill=gray!30!white ]   plot[smooth, domain=-1.7:1.7] (2* \x * \x-1, \x);

  \draw[black, line width = .3mm, ]   plot[smooth, domain=0:2.7] ( 2.5*\x-2 , \x);
  \draw[black, line width = .3mm, ]   plot[smooth, domain=0:2.7] ( 2.5*\x-2 , -\x);

  \draw[black,   thick, ->] (-5, 0) -- (6, 0) node[below]{$x$};
  \draw[black,  thick, ->] (0, -3) -- (0, 3) node[left]{$y$};
\end{tikzpicture}
\end{center}

\begin{conj}
Let $M$ be any complete noncompact Riemannian manifold with the exponential rate of volume growth $\gamma<\infty$ (see Section ~\ref{S4} for the definition).
Assume that the Weitzenb\"ock curvature on $k$-forms has a lower bound for some $0\leq k\leq n$.
Then there exist two real numbers $c\leq 0$ and $a>0$, such that  $(\Delta-w)^{-1}$ is bounded on $L^p (\Lambda^k(M))$  for any complex number $w$ which satisfies
\[
{\rm Re} (w)\leq a|{\rm Im} (w)|-c.
\]
\end{conj}

\vspace{.1in}

 In other words, if we drop the condition on the lower bound for the Ricci curvature, we believe that a generalization of Theorem \ref{thm3} should hold, with the spectrum contained in a strict sector rather than a parabola. In particular, if $k=0$ there is no condition on curvature, and the result becomes conformally invariant when the conformal factor is bounded. \\

In Theorem \ref{thm41-2} we extend the result  of Theorem~\ref{thm3} to a more general class of operators that are functions of the Laplacian.\\

 We conclude by providing a detailed description of the $L^p$-spectrum for $k$-forms over hyperbolic space $\mathbb H^{N+1}$, using Theorem~\ref{thm3} and Proposition \ref{kspec}. Let \begin{equation}\label{m-1-1}
 Q_{p,k} =\left\{ \left(\tfrac N2-k\right)^2+z^2 \, \Big| \; z\in \mathbb{C} \ \ \text{with} \ \ |\textup{Im}\, z|\leq \left|\tfrac Np -\tfrac N2\right| \right\}
\end{equation}
for $k\leq (N+1)/2$ and $1\leq p\leq\infty$, and let $Q_{p,k}=Q_{p,N+1-k}$ if $k>(N+1)/2$.

\begin{thm} \label{thm1}
The   spectrum  $\sigma(p,k)$ of the Laplacian on  $L^p$-integrable $k$-forms over the hyperbolic space $\mathbb{H}^{N+1}$  is given by
\[
\begin{split}
&\sigma(p,k)=Q_{p,k},   \qquad \qquad \qquad \,{\rm for}  \ \ 0\leq k \leq N/2; \qquad {\rm and} \\
&\sigma(p,\tfrac{N+1}{2})=Q_{p, \tfrac{N+1}{2}}\cup\{0\}, \quad {\rm if\,\, } N\,\, {\rm is \,\,odd}.
\end{split}
\]
 For $(N+1)/2\leq k \leq N+1$,
 $\sigma (p,k)=\sigma (p,N+1-k)$.

Moreover, for $p>2$  every point in the interior of the parabola $Q_{p,k}$ in \eref{m-1-1} is an eigenvalue for the Laplacian on  $L^p(\Lambda^k(\mathbb{H}^{N+1}))$, whereas for $1\leq p\leq 2$, none of the points in $Q_{p,k}$ is an eigenvalue.
\end{thm}

 We also prove the following interesting result for harmonic forms.
\begin{thm} \label{thm62}
Let $N$ be a positive odd integer. Then an $L^p$-integrable harmonic $(N+1)/2$-form exists over $\mathbb{H}^{N+1}$ if and only if
\[
p>\frac{2N}{N+1}.
\]
\end{thm}

These results are proved in Sections \ref{S3} and \ref{S7}.

 Note that when $N$ is an odd number, and if $p\leq 2N/(N+1)(<2)$, or $p\geq 2N/(N-1)$, then $0\in Q_{p, \tfrac{N+1}{2}}$; hence $\sigma(p,\tfrac{N+1}{2})=Q_{p, \tfrac{N+1}{2}}$. Combining this with the result of Theorem \ref{thm62}, we get the interesting observation that for $p\leq 2$ an $L^p$-integrable harmonic $(N+1)/2$-form exists only for the $p$ for which $\{0\}$ is an isolated point in the $L^p$-spectrum.


\vspace{0.1in}

{\bf Acknowledgement.} The authors would like to thank Michael Taylor for the clarifying remarks regarding the parametrix method.

\section{Resolvent kernel estimates} \label{S2}
We assume that $M$ is a complete Riemannian manifold of dimension $n$. Throughout
this paper, we let $0\leq k\leq n=\dim M$ be an integer.  In this section, we provide pointwise and integral estimates for the kernel of the resolvent $(\Delta+\alpha)^{-m}$, as well as its first-order derivatives, where $\alpha,m>0$  are positive numbers. Our estimates are new to the best of our knowledge. In particular, we believe that Corollary~\ref{cor211} is new and  would be useful in a broader context.

 Throughout the paper,
we use $\langle \cdot, \cdot \rangle$ to denote the inner product on vector fields and $k$-forms, and we denote the norm associated with this inner product by $|\cdot|$.  We will use  $\int_M f$ to denote the integral of a function over the Riemannian manifold without specifying the volume form if the context is clear; moreover, the constants  $C, C',C'', C_1,C_2, \ldots$ may differ from line to line even if they have the same superscript or subscript.

\begin{lem} \label{lemheat}
Let $M$ be a complete manifold of dimension $n$ with Ricci curvature bounded below ${\rm Ric}\geq -K_1$ and Weitzenb\"ock tensor on $k$-forms bounded below $\mathcal{W}_k\geq -K_2$. Then the heat kernel on $k$-forms, $\vec{h}_t$, has  the following  Gaussian upper bound
\[
|\vec{h}_t(x,y)| \leq C_1  \, {\rm vol}(B_x(\sqrt t))^{-1/2}{\rm vol}(B_y(\sqrt t))^{-1/2} e^{\sqrt{K_1 t} + K_2 t} \; e^{-\frac{d^2(x,y)}{C_2 t}},
\]
where ${\rm vol}(B_x(r))$ denotes the Riemannian volume of the geodesic ball or radius $r$ at $x$.

If in addition the  Weitzenb\"ock tensor on $(k+1)$-forms is bounded below, $\mathcal{W}_{k+1}\geq -K_2$, then
\[
|d\vec{h}_t(x,y)|\leq
C_1  \, t^{-1/2} \, {\rm vol}(B_x(\sqrt t))^{-1/2}{\rm vol}(B_y(\sqrt t))^{-1/2} e^{\sqrt{K_1 t} +2K_2 t} \; e^{-\frac{d^2(x,y)}{C_2 t}}.
\]
Similarly, if the  Weitzenb\"ock tensor on $(k-1)$-forms is bounded below,  $\mathcal{W}_{k-1}\geq -K_2$, then we have
\[
|\delta\vec{h}_t(x,y)|\leq
C_1  \, t^{-1/2} \,  {\rm vol}(B_x(\sqrt t))^{-1/2}{\rm vol}(B_y(\sqrt t))^{-1/2} e^{\sqrt{K_1 t} + 2K_2 t} \; e^{-\frac{d^2(x,y)}{C_2 t}}.
\]
In all of the above $C_1, C_2$ denote generalized constants that only depend on $n$ and possibly the curvature bounds.
\end{lem}

\begin{proof}  These estimates are known for the case of functions. The proof for the case of forms is similar, but we include it here for the sake of completeness.  When the Weitzenb\"ock tensor on $k$-forms is bounded below, using the domination technique as in \cite{HSU}, one can show that the heat kernel on functions dominates the heat kernel on $k$-forms  (see also \cites{CharJFA,DLi,Ros})
\begin{equation}\label{kheat2}
|\vec{h}_t(x,y) | \leq e^{K_2 t}  h(t,x,y)
\end{equation}
where $h(t,x,y)$ is the heat kernel of the Laplacian on functions.

By  \cite{SaCo}*{Theorem 6.1}, using the Ricci curvature lower bound, we get
\begin{equation}\label{kheat}
|h(t,x,y) |\leq C_1  \, {\rm vol}(B_x(\sqrt t))^{-1/2}\text{vol}(B_y(\sqrt t))^{-1/2} e^{\sqrt{K_1 t}}   \; e^{-\frac{d^2(x,y)}{C_2t}},
\end{equation}
where $C_1, C_2$ only depend on $n$.  The estimate for $\vec h_t$ follows.

Recall that the Hodge Laplacian on $k$-forms over $M$ may  be written as
\begin{equation}\label{bochner}
\Delta = \nabla^*\nabla +\mathcal{W}_k
\end{equation}
where $\nabla^*\nabla$ is the covariant Laplacian and $\mathcal{W}_k$ is the Weitzenb\"ock tensor. Using~\eqref{bochner}, we have
\begin{equation*}
 \frac 12 \Delta (|\omega|^2) = \langle \Delta \omega, \omega \rangle-|\n \omega|^2 -  \langle \mathcal{W} \omega, \omega \rangle.
\end{equation*}
By combining the  above formula for $(k+1)$-forms  together with Kato's inequality,    $|\n |\omega| \,|^2 \leq |\n \omega|^2$,  we get that $|d\vec{h}_t(x,y)|$ satifies
\begin{equation} \label{heat2}
\Delta|d\vec{h}_t(x,y)| + \partial/\partial t |d\vec{h}_t(x,y)| \leq K_2 |d\vec{h}_t(x,y)|,
\end{equation}
 whenever the Weitzenb\"ock tensor on $(k+1)$-forms is bounded below by $-K_2$.
As a result,
\begin{equation}
\Delta(e^{-K_2 t}|d\vec{h}_t(x,y)|) + \partial/\partial t (e^{-K_2 t}|d\vec{h}_t(x,y)| )\leq 0.
\end{equation}
Similarly, $|\delta\vec{h}_t(x,y)|$ satifies the same inequality, whenever  the Weitzenb\"ock tensor on $(k-1)$-forms is bounded below by $-K_2$.

 As in the proof in \cite{CLY}*{Theorem 6}, we can first obtain $L^2$-estimates for $|d\vec{h}_t|$ and $|\delta\vec{h}_t|$ using the $L^2$ estimates  for $|\vec{h}_t|$ and $h$. Then, the upper estimates follow by using the parabolic version of the Moser iteration as in \cite{SaCo}*{Theorems 5.1, 6.5} and the proof of Theorem 6 in \cite{CLY}.  \end{proof}

\begin{corl}\label{cor211}
 Let $x,y\in M$ be two distinct points. Under the same respective assumptions (depending on the order of the form) as in Lemma
\ref{lemheat}, we have
\[
|\vec{h}_t(x,y)| \leq C_1  \,  {\rm vol}(B_y(\sqrt t))^{-1 } e^{ C_3t} \; e^{-\frac{d^2(x,y)}{C_2 t}},
\]
and
\[
\begin{split}
& |d\vec{h}_t(x,y)| \leq C_1  \,  t^{-1/2}{\rm vol}(B_y(\sqrt t))^{-1 } e^{ C_3t} \; e^{-\frac{d^2(x,y)}{C_2 t}};\\
& |\delta\vec{h}_t(x,y)| \leq C_1  \,  t^{-1/2}{\rm vol}(B_y(\sqrt t))^{-1 } e^{ C_3t} \; e^{-\frac{d^2(x,y)}{C_2 t}}.
\end{split}
\]
\end{corl}

\begin{proof}
From the volume comparison theorem, we have
\begin{equation} \label{vol2-4}
{\rm vol}(B_x(\sqrt t))\geq C \,{\rm vol}(B_x(\sqrt t+d(x,y))) \left(1+\frac{d(x,y)}{\sqrt t}\right)^{-n} e^{-C'(t+d(x,y))},
\end{equation}
and since  ${\rm vol}(B_x(\sqrt t+d(x,y)))\geq  {\rm vol}(B_y(\sqrt t))$, we obtain
\begin{equation}\label{vol2-1}
{\rm vol}(B_x(\sqrt t))^{-1/2}\leq C \,{\rm vol}(B_{y}(\sqrt t))^{-1/2}\left(1+\frac{d(x,y)}{\sqrt t}\right)^{n/2}e^{C'(t+d(x,y))}
\end{equation}
 for any $x, y \in M$. For the rest of the paper,
we shall repeatedly use the following essential inequality:
for any $\sigma>0$, we have
\begin{equation}\label{impo}
-\frac{d(x,y)^2}{4C_2t} -  \sigma^2 t\leq - C_2^{-1/2}\sigma d(x,y) .
\end{equation}
 This inequality will allow us to use the Gaussian term in our estimates in order to generate an exponentially decaying term with respect to the distance $d(x,y)$, plus an additional exponentially increasing term in time.

By Lemma~\ref{lemheat} and  for a possibly larger $C_2$, the above estimates and \eqref{impo} yield
\[
|\vec{h}_t(x,y)| \leq C_1  \,  {\rm vol}(B_y(\sqrt t))^{-1 } e^{ C_3t} \left(1+\frac{d(x,y)}{\sqrt t}\right)^{n/2}\; e^{-\frac{d^2(x,y)}{C_2 t}}.
\]
Since the function $(1+z)^{n/2} e^{-z^2/c}$ is bounded for any $c>0$, the estimate of the corollary follows by taking a possibly larger $C_2$ once again.

The estimates for $d\vec{h}_t(x,y)$ and $\delta\vec{h}_t(x,y)$ follow by using the same method.
 \end{proof}

\begin{lem} \label{lem33}
Let $0\leq k\leq n$.
Suppose that an operator $T$ acting on $L^2(\Lambda^k(M))$ has integral kernel $\vec{k}(x,y)$ which satisfies
\begin{equation}\label{assum-3}
\sup_{x} \int_M |\vec{k}(x,y)|^{r^*} dy+\sup_{y} \int_M |\vec{k}(x,y)|^{r^*} dx \leq C
\end{equation}
 for some $1\leq r^* \leq \infty$.

Then, for any $1\leq p, q, r^* \leq \infty$ such that $ {\displaystyle 1+\frac{1}{q} =  \frac{1}{p} + \frac{1}{r^*}}$, the operator $T$ is bounded from $L^p(\Lambda^k(M))$ to $L^q(\Lambda^k(M))$.
\end{lem}

The above Lemma is known as Schur's test, or Young's inequality  and a proof when $T$ is an operator on $L^2(\mathbb{R}^n)$ can be found in~\cite{Sog}*{Theorem 0.3.1}.  For our case, the proof is identical since
\[
\|T\omega\|_{L^q}  \leq \left( \int_M\left|\int_M |\vec{k}(x,y)|\cdot|\omega(y)|  \, dy\right|^q dx \right)^{1/q},
\]
 and as a result all of the  H\"older estimates in \cite{Sog} can easily be extended to our operator by our assumption on the integrability of the kernel.  \\

Note that the condition on $p, q, r^*$ implies that the above result only works for $q \geq p$.

Using the above lemma, we will now obtain the $L^p$ to $L^q$ boundedness for various operators related to the resolvent  operator $(\Delta+\alpha)^{-m}$, where $\alpha\gg 0$ and $m>0$. Some of them will be used in the following section for the proof of the $L^p$-Weyl criterion,  but also in Section \ref{S4} for the proof of Theorem \ref{thm3}. For reasons that will become apparent in the upper estimate, we change our notation a bit and denote $\alpha=\xi^2$. For any $m>0$ and any $\xi \in \mathbb{R}$ we denote by $\vec{g}_{m,\xi}(x,y)$ the kernel of the resolvent operator $(\Delta+\xi^2)^{-m}$.

\begin{thm} \label{corl31}
Suppose that $M$ is a complete noncompact manifold with Ricci curvature bounded below.   Let
\[
\phi(x)=\frac{1}{\sqrt{{\rm vol}(B_x(1))}}.
\]
For any $\alpha>0$, $1\leq r^*\leq \infty$, let $r$ be the conjugate number to $r^*$, that is $1/r+1/r^*=1$.
Let $p,q,r^*$ as in Lemma \ref{lem33}. If  the Weitzenb\"ock tensor on $k$-forms is bounded below and
\[
m>\frac{n}{2r},
\]
then  the operators
\[
\begin{split}
& T=\phi^{1/r^*-1}(\Delta + \xi^2)^{-m} \phi^{1/r^*-1};\\
& T'=\phi^{2/r^*-2}(\Delta + \xi^2)^{-m} ;\\
& T''= (\Delta + \xi^2)^{-m} \phi^{2/r^*-2}.
\end{split}
\]
  are bounded from $L^p$ to $L^q$ for any $\xi$ large enough.

If
\[
m>\frac{n}{2r}+\frac 12,
\]
 and the Weitzenb\"ock tensor on $(k+1)$-forms is bounded below  then the operators
 \[
\begin{split}
& T_1=\phi^{1/r^*-1}d(\Delta + \xi^2)^{-m} \phi^{1/r^*-1};\\
& T'_1=\phi^{2/r^*-2}d(\Delta + \xi^2)^{-m} ;\\
& T''_1= d(\Delta + \xi^2)^{-m} \phi^{2/r^*-2}
\end{split}
\]
 are bounded from $L^p$ to $L^q$ for any $\xi$ large enough.

 If
\[
m>\frac{n}{2r}+\frac 12,
\]
and the Weitzenb\"ock tensor on $(k-1)$-forms is bounded below then the operators
 \[
\begin{split}
& T_2=\phi^{1/r^*-1}\delta(\Delta + \xi^2)^{-m} \phi^{1/r^*-1};\\
& T'_2=\phi^{2/r^*-2}\delta(\Delta + \xi^2)^{-m} ;\\
& T''_2= \delta(\Delta +\xi^2)^{-m} \phi^{2/r^*-2}.
\end{split}
\]
are bounded from $L^p$ to $L^q$ for any $\xi$ large enough.
\end{thm}

\begin{proof}
We only provide the proof for the case of $T$, since the remaining cases are proved in a similar manner.  The additional $+1/2$ in the condition for $m$ comes from the additional  $t^{-1/2}$ factor in the upper estimates for the derivatives of the heat kernel.

The kernel $\vec k_{m,\xi}(x,y)$ of $T$ is given by
\[
\vec k_{m,\xi}(x,y)=\phi(x)^{\frac{1}{r^*}-1}\phi(y)^{\frac{1}{r^*}-1}\vec g_{m,\xi}(x,y),
\]
where
\begin{equation} \label{resol}
\vec g_{m,\xi}(x,y)=c_n\int_0^\infty t^{m-1}e^{-\xi^2 t} \,\vec h_{t}(x,y)\, dt.
\end{equation}
By the volume estimate, we have
\[
\frac{{\rm vol}(B_x(1))}{{\rm vol}(B_y(1))} \leq \frac{{\rm vol}(B_y(1+d(x,y)))}{{\rm vol}(B_y(1))} \leq C\, (1+d(x,y))^n \, e^{C' d(x,y)} \leq C\,  e^{C' d(x,y)} .
\]
Bearing in mind that $1-r^*\leq 0$,  we therefore have
\begin{equation}\label{reso0}
|\vec k_{m,\xi}(x,y)|^{r^*}\leq \,C_1 \, \phi(y)^{2- 2r^*}   \, e^{C' d(x,y)} \, |\vec g_{m,\xi}(x,y)|^{r^*}.
\end{equation}

Substituting the upper estimate for the heat kernel from Corollary \ref{cor211} we get
\[
\begin{split}
|\vec k_{m,\xi}(x,y)|^{r^*}&\leq \,C_1 \,  \phi(y)^{2 -2r^*}   e^{C' d(x,y)} \left( \int_0^\infty t^{m-1} {\rm vol}(B_y(\sqrt t))^{-1 }  e^{-\xi^2 t}  e^{ C_3t} \; e^{-\frac{d^2(x,y)}{C_2 t}} \; dt \right)^{r^*}\\
&\leq  \,C_1 \,  \phi(y)^{2 -2r^*}   \left( \int_0^\infty t^{m-1} {\rm vol}(B_y(\sqrt t))^{-1 }  e^{- \frac 12 \xi^2 t}  \; e^{-\frac{d^2(x,y)}{C_2 t}} \; dt \right)^{r^*}
\end{split}
\]
for all $\xi$ large enough, after using \eqref{impo}.
Fixing $y$, let
\begin{equation}\label{af}
f(a)={\rm vol}(B_y(a ))
\end{equation}
for $a>0$. Then $f(1)=\phi(y)^{-2}$. Using the co-area formula, we obtain
\[
\int_M |\vec k_{m,\xi}(x,y)|^{r^*}dx\leq C_1\int_0^\infty   \, f'(\rho) \left( \int_0^\infty \frac{f(1)^{1/r}}{f(\sqrt t)}   \; t^{m-1}  e^{-\frac 12  \xi^2 t}   \; e^{-\frac{\rho^2}{C_2 t}} \, dt \right)^{r^*}d\rho .
\]
By volume comparison, $ f(1)/f(\sqrt t)\leq C \max\left(1, t^{-n/2}\right)$, hence, after splitting $f(\sqrt t)= f(\sqrt t)^{1/r} f(\sqrt t)^{1/r^*},$ we get
\[
\int_M |\vec k_{m,\xi}(x,y)|^{r^*}dx\leq C_1\int_0^\infty   \, f'(\rho) \left( \int_0^\infty  f(\sqrt t)^{-1/r^*}  \, d\mu \right)^{r^*}d\rho .
\]
where $d\mu =t^{m-1} \max\left(1, t^{-n/2r}\right)  e^{-  \frac 12   \xi^2 t}   \; e^{-\frac{\rho^2}{C_2 t}} \, dt$.  By the H\"older inequality,
\[
\begin{split}
&\left( \int_0^\infty   f(\sqrt t)^{-1/r^*}     d\mu \right)^{r^*} \leq \left( \int_0^\infty  d \mu \right)^{r^*/r} \cdot \int_0^\infty    f(\sqrt t)^{-1}   \;   d\mu.
\end{split}
\]

Note that
\[
\int_0^\infty  d \mu \leq C
\]
whenever $m> n/2r$, therefore
\[
\begin{split}
\int_M |\vec k_{m,\xi}(x,y)|^{r^*}dx&\leq C_1\int_0^\infty   \int_0^\infty   \, \frac{f'(\rho)}{f(\sqrt t)}  \max\left(1, t^{-n/2r}\right)  \; t^{m-1}  e^{-\frac 12   \xi^2 t}   \; e^{-\frac{\rho^2}{C_2 t}} \, dt\, d\rho \\
&= C_1\int_0^\infty   \int_0^\infty   \, \frac{f(\rho)}{f(\sqrt t)}  \,\rho \, \max\left(1, t^{-n/2r}\right)  \; t^{m-2}  e^{-  \frac 12 \xi^2 t}   \; e^{-\frac{\rho^2}{C_2 t}}  \, d\rho \, dt
\end{split}
\]
after applying Fubini's theorem and integration by parts. We estimate,
\[
\int_0^{\sqrt t}  \frac{f(\rho)}{f(\sqrt t)}  \,\rho \, e^{-\frac{\rho^2}{C_2 t}}  \, d\rho  \leq  \int_0^{\sqrt t} \,\rho   \; e^{-\frac{\rho^2}{C_2 t}}  \, d\rho \leq  \int_0^{\infty} \,\rho   \; e^{-\frac{\rho^2}{C_2 t}}  \, d\rho \leq C \, t
\]
and
\[
\begin{split}
\int_{\sqrt t}^\infty  \frac{f(\rho)}{f(\sqrt t)}  \,\rho \; e^{-\frac{\rho^2}{C_2 t}}  \, d\rho  \leq  \int_0^{\infty} \,\rho   \, \left(\frac{\rho}{\sqrt t}\right)^{n}\; e^{C' \rho}\; e^{-\frac{\rho^2}{C_2 t}}  \, d\rho \leq   C  t\, e^{C' t}
\end{split}
\]
 where we have employed similar methods as in the proof of Corollary \ref{cor211} to get the upper bounds,  including \eqref{impo} and the fact that $z e^{-z^2/c}$ is bounded for any $c>0$.  Combining the above, we have
\begin{equation*}\label{reso2}
\begin{split}
\int_M |\vec k_{m,\xi}(x,y)|^{r^*}dx&\leq C_1  \int_0^\infty   \,  \max\left(1, t^{-n/2r}\right)  \; t^{m-1}   \, e^{C' t} \,e^{- C_4  \xi^2 t}   \; dt \leq C
\end{split}
\end{equation*}
whenever $m>n/2r$,  and for all $\xi$ large enough.

\end{proof}

\begin{corl} \label{corl32}
Suppose that $M$ is a complete noncompact manifold with Ricci curvature and the Weitzenb\"ock tensor on $k$-forms both bounded below. Let $m>n/4$ and $C_5>0$. Then  for any $\xi$ large enough
\begin{equation}\label{upper-e}
\int_M\phi(y)^{-2}   e^{C_5 d(x,y)} |\vec g_{m,\xi}(x,y)|^{2}dx\leq C_6
\end{equation}
and  for $C_7>0$ small enough\footnote{$C_7$ depends only on $C_2$ from the Gaussian estimate.}
\begin{equation}\label{upper-e2}
\int_M\phi(y)^{-2}   e^{C_7\xi d(x,y)} |\vec g_{m,\xi}(x,y)|^{2}dx\leq C_6.
\end{equation}
\end{corl}
\begin{proof}
In the proof of the previous theorem we have in fact shown that
\[
\int_M  C_1 \, \phi(y)^{2- 2r^*}   \, e^{C' d(x,y)} \, |\vec g_{m,\xi}(x,y)|^{r^*} \leq C_6
\]
for any $C'>0$, $\xi$ large enough and $m>n/2r$. Replacing $C'$ by $C_7\xi$, we note that $e^{C_7\xi d(x,y)}$ may also be absorbed by the Gaussian term when applying \eqref{impo} in the proof of the theorem  if we choose $C_7$ so that $-\xi^2+(C_7\xi)^2 C_2 \leq -C_4 \xi^2$ for some $C_4>0$. Hence,
\[
\int_M  C_1 \, \phi(y)^{2- 2r^*}   \, e^{C_7\xi d(x,y)} \, |\vec g_{m,\xi}(x,y)|^{r^*} \leq C_6
\]
for all $\xi$ large enough. The Corollary follows by taking $r^*=2$.
\end{proof}

 \begin{corl} \label{lem31}
 Let $m>0$. Under the same respective assumptions (depending on the order of the form) as in Theorem~\ref{corl31}, we have that
$\vec{g}_{m,\xi}(x,y)$ is $L^1$-integrable with a uniform upper bound  for $\xi$ large enough. Therefore the operator $(\Delta+\alpha)^{-m} $  is bounded on $L^1(\Lambda^k(M))$  for $\alpha(=\xi^2)$ large enough.

Similarly, for $m>1/2$,
 $d\vec{g}_{m,\xi}(x,y)$,  and $\delta\vec{g}_{m,\xi}(x,y)$ are  $L^1$-integrable with a uniform upper bound  for $\xi$ large enough. Therefore the operators  $d(\Delta+\alpha)^{-m} $ and $\delta (\Delta+ \alpha )^{-m}$ are also bounded on $L^1(\Lambda^k(M))$  for $\alpha$ large enough.
\end{corl}

\begin{proof} This follows from Theorem \ref{corl31} by letting $r^*=1$. However, we can provide a shorter proof for the $L^1$-boundedness of $(\Delta+\alpha)^{-m}$.

 Let $\alpha=\xi^2$. Given that the  Weitzenb\"ock tensor is bounded below,  estimate \eref{kheat2} and formula \eqref{resol}  imply that
\[
|\vec{g}_{m,\xi}(x,y)|\leq c_n \int_0^\infty e^{-(\xi^2-K_2)t}t^{m-1} h(t,x,y) dt.
\]
By \cite{GrHeat}*{Theorem 7.16}, the heat kernel of the Laplacian on functions satisfies
\[
\int_M h(t,x,y) dx\leq 1.
\]
Therefore,
\[
\int_M |\vec{g}_{m,\xi}(x,y)|dx\leq \int_0^\infty e^{-(\xi^2-K_2)t}t^{m-1}  dt\leq C<\infty
\]
whenever $m>0$ and $\xi^2>K_2$. This proves that the operator $(\Delta+\xi^2)^{-m}$ is bounded on $L^1(\Lambda^k(M))$ for $m>0$ and $\xi^2>K_2$.
\end{proof}

\begin{remark}
If we assume $m$ is sufficiently large, then the proofs of Theorem~\ref{corl31}, Corollary~\ref{corl32}, and Corollary~\ref{lem31} are standard  and follow from the well-known heat kernel and volume estimates.  See~\cite{sturm} for details. In our results, the number $m$ is optimal.
\end{remark}

\begin{remark}
We would like to remark that since  Lemma \ref{lem33} holds for a more general class of operators we could use it to obtain resolvent estimates as in Theorem \ref{corl31} for operators other than the Laplacian. For example, consider a self-adjoint operator $\mathcal{L}$ acting on $L^2$-integrable sections of bundles over Riemannian manifolds (with or without boundary). Whenever $\mathcal{L}$ can be extended to the space of $L^p$-integrable sections via the heat semigroup, then one can use the kernel of the operator on $L^2$ to obtain results for the boundedness of $\mathcal{L}$ between the various $L^p$ spaces  (see Section \ref{S0} for the case of the Laplacian on forms). For example, one could extend the above lemma to the square of Dirac operators on Clifford bundles over complete manifolds whenever the Clifford contraction and Ricci curvature are bounded below (see \cite{ChGr} for how the square Dirac operator can be extended to $L^p$ via its heat semigroup).
\end{remark}

\section{Weyl Criterion for   the  $L^p$-Spectrum} \label{S0}

In this section we will provide the detailed definition of the $L^p$-spectrum of the Laplacian on a complete manifold $M$ for $p\in[1,\infty]$.  We begin by recalling the extension of the Laplacian on $L^p(\Lambda^k(M))$ using the semi-group of the operator. For any $t\geq 0$, the heat operator  $e^{-t\Delta}$  is a bounded operator  on $L^2(\Lambda^k(M))$.  Whenever  $\mathcal{W}_k$ is bounded below, $e^{-t\Delta}$  extends to a  semigroup of operators on $L^p(\Lambda^k(M))$ for any $p\in[1,\infty]$, with $L^\infty(\Lambda^k(M))$ defined as $(L^1(\Lambda^k(M)))^*$ (see for example~\cite{Char2}).   The Laplacian on $L^p(\Lambda^k(M))$, $\Delta_p$, is defined as the infinitesimal generator of this semigroup when acting on $L^p$.   We will denote the domain of the infinitesimal generator $\Delta_p$   by ${\mathcal Dom}(\Delta_p)$, and for simplicity we will often refer to $\Delta_p$ as the Laplacian acting on $L^p(\Lambda^k(M))$.

Analogously to the $L^2$ case, the spectrum of $\Delta_p$ on $k$-forms consists of all points $\lambda\in \mathbb{C}$ for which $\Delta_p-\lambda I$ fails to be invertible on   $L^p$. The  essential spectrum of  $\Delta_p$ on $k$-forms, $\sigma_\textup{ess}(p,k)$, consists of the cluster points in the spectrum and of isolated eigenvalues of  infinite multiplicity. Both the spectrum and the essential spectrum are closed subsets of $ \mathbb{C}$. Whenever the domain is clear from the context we will denote the Laplacian on $L^p(\Lambda^k(M))$ by $\Delta$, even though it varies with $p$ and $k$.

For $p$ and $p^*$ dual dimensions such that $1/p + 1/p^* =1$, it is well known that $L^p(\Lambda^k(M))$ and $L^{p^*}(\Lambda^k(M))$ are dual spaces, and   $\Delta_p$ and $\Delta_{p^*}$ are dual operators. This  implies that $\sigma(p,k)=\sigma(p^*,k),$ and the same is true for the essential spectrum. However, the nature of the points in the $L^p$ and $L^{p^*}$ spectrum can significantly differ. For example, Taylor  proved that over symmetric spaces of noncompact type, every point inside a parabolic region is an eigenvalue for the $L^p$-spectrum of the Laplacian on functions for $p>2$\cite{tay}. Similarly, Ji and Weber proved that over locally symmetric spaces of rank 1, the $L^p$-spectrum for $p\geq 2$  for the Laplacian on functions contains an open subset of $\mathbb{C}$ in which every point is an eigenvalue, whereas for $1< p < 2$ the set of $L^p$ integrable eigenvalues is a discrete set  \cites{JW1}. In the case of finite volume the opposite scenario can occur. In  \cites{JW2}, the same authors show that over certain non-compact locally symmetric space of higher rank and with finite volume every point in a parabolic region, except for a discrete set, is an eigenvalue for $1< p < 2$.  In Theorem \ref{thm1} we see that this absence in symmetry with regards to the spectrum also holds for the Laplacian on $k$-forms over the hyperbolic space $\mathbb{H}^{N+1}$. It is also reflected in the following analytic result, Theorem \ref{Weyl},  which illustrates why a Weyl criterion might not be possible on a general manifold. In Theorem \ref{thm2}, we will then find sufficient conditions on the manifold such that a Weyl criterion holds.  \begin{thm} \label{Weyl}
A complex number $\lambda$ is in the spectrum of $\Delta=\Delta_p$ on $k$-forms, if and only if one of the following holds:

$(a)$ For any $\eps>0$,
there is a $k$-form $\omega\in {\mathcal Dom}(\Delta_p)$  such that
\begin{equation}\label{approx-2}
\|\Delta\omega-\lambda\,\omega\|_{L^p}\leq \eps\|\omega\|_{L^p}
\end{equation}
or,

$(b)$ For any $\eps>0$,
there is a  $k$-form $\omega\in {\mathcal Dom}(\Delta_{p^*})$ such that
\[
\|\Delta\omega-\lambda\,\omega\|_{L^{p^*}}\leq \eps\|\omega\|_{L^{p^*}},
\]
where $p^*$ satisfies $1/p + 1/{p^*} =1$.
\end{thm}
\begin{proof}
It is clear that if $(a)$  holds  then $\lambda$ is in the $L^p$-spectrum by definition. If $(b)$ holds, then $\lambda$ is again in the $L^p$-spectrum since  the $L^p$-spectrum and the $L^{p^*}$-spectrum are equal.

For the converse, suppose that neither $(a)$ nor $(b)$ hold for $\lambda$. Then, there exists $\eps_o>0$ for all  $\omega\in {\mathcal Dom}(\Delta_p)$
\[
(i) \ \ \|(\Delta -\lambda)\,\omega\|_{L^p}\geq \eps_o\|\omega\|_{L^p}
\]
and  for all  $\omega\in {\mathcal Dom}(\Delta_{p^*})$
\[
(ii) \|(\Delta -\lambda)\,\omega\|_{L^{p^*}}\geq \eps_o\|\omega\|_{L^{p^*}}.
\]
We claim that the set
\[
A = \{ \ (\Delta-\lambda)\,\omega \ \Big| \ \omega \in {\mathcal Dom}(\Delta_p) \}
\]
is dense in ${\mathcal Dom}(\Delta_p)$.  If not, then there exists a non-zero $k$-form $\eta \in {\mathcal Dom}(\Delta_{p^*})$ such that
\[
\left( (\Delta-\lambda)\omega, \eta \right) =0
\]
for all $\omega \in {\mathcal Dom}(\Delta_p)$, where $(\,\, , \,\,)$ denotes the $(L^p,L^{p^*})$-pairing. It easily follows that    $(\Delta-\lambda)\eta =0$ which contradicts $(ii)$.

By $(i)$, $(\Delta -\lambda)^{-1}$ is defined on $A$, and we have
\[
\|(\Delta -\lambda)^{-1}(\Delta-\lambda)\,\omega\|_{L^p} =\| \omega\|_{L^p} \leq \tfrac{1}{\eps_o} \|(\Delta-\lambda)\,\omega\|_{L^p}
\]
for any $\omega \in A$. It follows that $(\Delta -\lambda)^{-1}$ is bounded on $A$, and the proposition follows since $A$ is dense in ${\mathcal Dom}(\Delta_p)$.
\end{proof}

When $p=2$, the two conditions in Theorem \ref{Weyl}  are identical, and the theorem is reduced to the classical Weyl criterion.  It is then natural to ask under which additional conditions on the manifold would~\eqref{approx-2} suffice for $p\neq 2$.     We will use the estimates of the previous section to obtain the following  $L^p$-Weyl criterion for the spectrum.

\begin{thm} \label{thm2}
Suppose that $M$ is a complete open manifold with Ricci curvature and the Weitzenb\"ock tensor on $k$-forms bounded below, and such that volume of geodesic balls of radius one is uniformly bounded below. Fix $p \geq 2$. Then,  a complex number $\lambda$ is in the spectrum of $ \Delta=\Delta_p$  on $k$-forms if and only if for any $\eps>0$
there is a $k$-form $\omega\in  {\mathcal Dom}(\Delta_p)$  such that
\[
\|\Delta\omega-\lambda\,\omega\|_{L^p}\leq \eps\|\omega\|_{L^p}.
\]
\end{thm}
\begin{proof}
By Theorem \ref{Weyl} it suffices to prove the converse. Suppose that  for some $\lambda$ there exists $\eps_o>0$ such that for all  $\omega\in  {\mathcal Dom}(\Delta_p) $
\begin{equation} \label{2_e2}
\|(\Delta -\lambda)\,\omega\|_{L^p}\geq \eps_o\|\omega\|_{L^p}
\end{equation}
Similarly to the  proof or Theorem \ref{Weyl} we consider the set
\[
A = \{ \ (\Delta-\lambda)\,\omega \mid\omega \in  {\mathcal Dom}(\Delta_p)  \ \}
\]
and suppose that it is dense in $ {\mathcal Dom}(\Delta_p)$. If not, there exists a non-zero $k$-form $\eta \in L^{p^*}$ with the properties that $\eta\in  {\mathcal Dom}(\Delta_{p^*})$ and such that $\Delta \eta =\lambda\eta$.  Then for any $\alpha >0$, and positive integer $m,$
\[
(\Delta + \alpha)^{-m} \eta =\frac{1}{(\lambda +\alpha)^m} \eta.
\]
 By Theorem \ref{corl31},  taking $r^*=p/2$, we get that the operator  $\phi^{4/p-2}(\Delta + \alpha)^{-m}$   is bounded from $L^{p^*}(\Lambda^k(M))$ to $L^p(\Lambda^k(M))$  for $\alpha$ large enough and $m>n/2$. Given the uniform lower bound on $\phi$,    $(\Delta + \alpha)^{-m}$ is  also bounded from $L^{p^*}(\Lambda^k(M))$ to $L^p(\Lambda^k(M))$. It follows that $\eta \in L^p(\Lambda^k(M))$, and  consequently, $\eta\in  {\mathcal Dom}(\Delta_p)$. By \eref{2_e2},
\[
\|(\Delta -\lambda)\,\eta\|_{L^p}\geq \eps_o\|\eta\|_{L^p}
\]
and hence $\eta=0$ since $\Delta \eta =\lambda\eta$, which is a contradiction.
\end{proof}

 Our final result for this section is the relationship between the $L^p$-spectra for three consecutive orders of forms.
\begin{prop}\label{kspec}
Suppose that $M^n$ is a complete manifold with Ricci curvature bounded below. Set $0\leq k \leq n-1$ and suppose that the Weitzenb\"ock tensor on $(k-1)$, $k$ and $(k+1)$-forms  is bounded below over $M$.  Then, for any  $1\leq p \leq \infty$
\[
\sigma (p,k)  \setminus\{0\} \subset \sigma (p,k-1) \cup \sigma (p,k+1)
\]
The same result also holds for the essential spectrum\footnote{We use the notation $\sigma (p,-1) =\emptyset$  and that the Weitzenb\"ock tensor on $(-1)$-forms is trivially bounded below.}.
\end{prop}
\begin{proof} By Corollary \ref{lem31} the resolvent operator  $(\Delta + \alpha)^{-s}$ is bounded on $L^1(\Lambda^k(M))$ for any $s>0$, provided that $\alpha$ is large enough.

Set $\lambda\in \mathbb{C}\setminus 0$ such that $\lambda \in \rho (1,k-1) \cap \rho (1,k+1)$, where $\rho(p,k)=\mathbb{C} \setminus \sigma (p,k)$ is the resolvent set.  By the resolvent equation,  for any positive integer $m$
\[
\begin{split}
&
(\Delta-\lambda)^{-1}=(\Delta+\alpha)^{-1}+(\lambda+\alpha)(\Delta+\alpha)^{-2}
+\cdots+(\lambda+\alpha)^{m-1}(\Delta+\alpha)^{-m}\\
&-\frac{(\lambda+\alpha)^m}{\lambda}(\Delta+\alpha)^{-m}+\frac{(\lambda+\alpha)^m}{\lambda}(\Delta-\lambda)^{-1}\Delta(\Delta+\alpha)^{-m}.
\end{split}
\]
Thus $(\Delta-\lambda)^{-1}$ is bounded on $L^1(\Lambda^k(M))$ if and only if   $(\Delta-\lambda)^{-1}\Delta(\Delta+\alpha)^{-m}$ is bounded on $L^1(\Lambda^k(M))$. Since $m$ is arbitrary, we replace $m$ by $2m$ and write
\begin{equation} \label{res1}
\begin{split}
 &  (\Delta -\lambda )^{-1} \Delta (\Delta + \alpha)^{-2m}     \\
     &  \ \  =d (\Delta+\alpha)^{-m}  (\Delta -\lambda )^{-1} \delta (\Delta+\alpha)^{-m} +\delta
(\Delta+\alpha)^{-m}  (\Delta -\lambda )^{-1}  d (\Delta+\alpha)^{-m}.
\end{split}
\end{equation}
By Corollary \ref{lem31} the operators  $d(\Delta+\alpha)^{-m} $ and $\delta (\Delta+\alpha)^{-m}$ are bounded on $L^1$ under our curvature assumptions, for any $m>1/2$ and $\alpha$  large enough. At the same time, if $\lambda\in  \rho (1,k-1) \cap \rho (1,k+1)$ we have that  the first resolvent operator on the right side of \eref{res1}, $(\Delta -\lambda )^{-1}  $, is bounded on $L^1(\Lambda^{k-1}(M))$, and the second is bounded on $L^1(\Lambda^{k+1}(M))$ (the first operator vanishes when $k=0$). As a result, the operator on the right side of \eref{res1} is bounded on $L^1(\Lambda^k(M))$, and in consequence $\lambda\in \rho (1,k)$ if $\lambda\neq 0$  by the resolvent equation above.

 The result is known on $L^2$ by  \cite{Cha-Lu-2}. Therefore, using an interpolation argument, we get that $ (\Delta+\alpha)^{-m} $, $d(\Delta+\alpha)^{-m} $ and $\delta (\Delta+\alpha)^{-m}$ are bounded on $L^p$  under our curvature assumptions for any large enough $\alpha>0$.  Using a similar argument as the one above we conclude that the containment is true for the $L^p$-resolvent as well.
\end{proof}

\begin{remark}
In a recent article, we have demonstrated that, independently of the curvature conditions on the manifold, the $L^2$-spectrum of the Laplacian on 1-forms should always contain the spectrum of the Laplacian on functions.  Proposition \ref{kspec} illustrates that this is also the case for the $L^p$ spectrum. At the same time, Theorem \ref{thm1} is consistent with the above Proposition. However, the unimodality with respect to $k$  that appears in the form essential spectrum of the hyperbolic space is not a general characteristic of all manifolds, as there exist 4-dimensional flat manifolds whose 2-form essential spectrum is smaller than their 1-form essential spectrum \cite{Cha-Lu-2}.
\end{remark}

\section{Finite Propagation Speed}

In this section, we will provide a   proof for the finite propagation speed of the wave equation for the Hodge Laplacian on forms.

For the case of the Laplacian on functions, we refer the interested reader to \cites{Jaf}. Sikora in \cite{Sik}*{Theorem 6} considered operators more general than the Hodge Laplacian on forms, but our simplified argument works well for  our case.

\begin{thm} \label{thmFP}
Let $M$ be a complete Riemannian manifold, and $H=\Delta - \gamma_o$ be a nonnegative operator on the space of $L^2$ integrable $k$-forms for some $\gamma_o\geq 0$. Then the wave kernel corresponding to $H$ has finite propagation speed which is at most 1. In particular,   if $G(t,x,y)$ denotes the kernel of the operator $\cos(t\sqrt{H})$, then $G(t,x,y)=0$ whenever $d(x,y) \geq t$.
\end{thm}

\begin{proof}
Fix $x,y \in M$  and let $\eps>0$. For $t<d(x,y)-\eps$ consider the cone
\[
C_{\eps} =\{ (z,t) \in M\times \mathbb{R}^+ \, \big| \; d(z,y) < d(x,y) -t-\eps \}.
\]
We will show  that $ G(t,x,z) =0$ for all $(z,t) \in C_\eps$.

Let $d=d(x,y)$ and denote by $B_y(r)$ the ball of radius $r$ centered at $y$ and by $\partial B_y(r)$ its boundary. Define
\[
E(t)= \frac 12 \int_{B_y (d-t-\eps)} \left[ \; |G_t|^2 + |dG|^2 +|\delta G|^2 +\gamma_o |G|^2 \; \right].
\]
Then
\begin{equation} \label{2_e1}
\begin{split}
E'(t) = & \int_{B_y (d-t-\eps)} \left[ \; \<G_{tt}, G_t \> + \<dG_t ,dG\> +\<\delta G_t, \delta G\>  + \gamma_o \< G_t, G\> \; \right]\\
&- \frac 12 \int_{\partial B_y (d-t-\eps)} \left[ \; |G_t|^2 + |dG|^2 +|\delta G|^2 +\gamma_o |G|^2 \; \right].
\end{split}
\end{equation}
Integration by parts gives
\[
\begin{split}
& \int_{B_y (d-t-\eps)} \left[ \; \< G_{tt}, G_t \> + \<dG_t , dG\> +\<\delta G_t, dG\>+ \gamma_o \<G_t, G\> \; \right] \\
&=\int_{B_y (d-t-\eps)} \left[ \; -\<\Delta G, G_t \> + \<dG_t , dG\> +\<\delta G_t, \delta G\> +2\gamma_o \<G_t, G\>   \; \right] \\
&= \int_{\partial B_y (d-t-\eps)} \left[ \; \<d G, \nu \wedge G_t \> - \< \nu \wedge \delta G, G_t\> \; \right]  +2\gamma_o  \int_{B_y (d-t-\eps)} \<G_t, G\> ,
\end{split}
\]
where $\nu$ is the co-normal to the outward unit vector field on the boundary.
By decomposing $G_t$ into its normal and tangential parts on the boundary and the Cauchy-Schwarz inequality we can prove that
\[
 \int_{\partial B_y (d-t-\eps)} \left[ \; \<d G, \nu \wedge G_t \> - \< \nu \wedge \delta G, G_t\>  \; \right] \leq \frac12 \int_{\partial B_y (d-t-\eps)} \left[ \; |G_t|^2+ |dG|^2 + |\delta G|^2 \; \right].
\]
 Substituting the above estimates into \eref{2_e1} we get
\[
E'(t)\leq  2 \gamma_o \int_{B_y (d-t-\eps)}  \<G_t, G \>  - \frac{\gamma_o}{2} \int_{\partial B_y (d-t-\eps)}    |G|^2 \leq  \sqrt{\gamma_o} E(t),
\]
where we have used the   Cauchy-Schwarz inequality  $2|\< G_t,  G \>|\leq \sqrt{\gamma_o^{-1}}|G_t|^2+\sqrt{\gamma_o}|G|^2 $. Since $E(0)=0$, and $\eps>0$ is arbitrary, we have that $E(t) = 0$ for all $t \leq d(x,y).$

\end{proof}

\section{The $L^p$-resolvent set} \label{S4}
We begin this section by defining the \emph{exponential rate of volume growth}  of a manifold $M$. The rate of volume growth of $M$, which we denote by $\gamma$, is the infimum of all real numbers satisfying the property: for any $\eps>0$, there is a constant $C(\eps)$, depending only on $\eps$ and the dimension of the manifold, such that for any $x\in M$ and any $R\geq 1$, we have
\begin{equation} \label{3-1a}
{\rm vol}(B_x(R))\leq C(\eps) {\rm vol}(B_x(1))e^{(\gamma+\eps)R}.
\end{equation}
$\gamma$ is defined to be $\infty$ if for any $\gamma>0$ and  any $C>0$, we can find a pair of $(x,R)$ such that the above inequality is reversed  with $C(\eps)=C$.

By the Bishop-Gromov Volume Comparison Theorem, if ${\rm Ric}(M)\geq -K$, then we have $\gamma\leq K$.
 The  exponential rate of volume growth  is $0$  for $\R^n$, and is $n-1$ for  the $n$-dimensional hyperbolic space.

 If the volume of the manifold is finite, then we have the following result
 \begin{prop}
 Let $\gamma$ be the exponential rate of volume growth of $M$ and suppose that $M$ has finite volume. Then for any $\eps>0$, there exists a constant $c(\eps)$, depending only on $\eps$ and the dimension of $M$, such that for any $x\in M$ and $R>1$,
\begin{equation}\label{3-1b}
{\rm vol}(M)-{\rm vol}(B_x(R))\geq c(\eps)\,{\rm vol}(B_x(1)) e^{-(\gamma+\eps)R}
\end{equation}
 \end{prop}

 \begin{proof}
 For any $R>1$, let $y\in M$ such that $d(x,y)=R+1$. Then since $M\backslash B_x(R))\supset B_y(1)$, we    have
\[
{\rm vol}(M)-{\rm vol}(B_x(R))\geq {\rm vol}(B_{y}(1))\]
By the definition of $\gamma$,  there is a $C(\eps)$ such that
\[
{\rm vol}(B_y(R+2))\leq C(\eps) {\rm vol}(B_y(1))e^{(\gamma+\eps)(R+2)}.
\]
Since $B_x(1)\subset B_y(R+2)$, we have
\[
{\rm vol}(B_x(1))\leq {\rm vol}(B_y(R+2)).
\]
The proposition follows by combining the above three inequalities.
 \end{proof}

By the above proposition, it is clear that in general, $\gamma\neq 0$ even if the volume of the manifold is finite. For example, let $M$ be a complete Riemannian manifold of constant curvature $-1$ which has finite volume. For  a fixed point $x_o$, since ${\rm vol}(B_x(1))\sim e^{-(n-1) \,d(x,x_o)}$, we have $\gamma=n-1\neq 0$.

\begin{corl}
Let $x_o\in M$ be a fixed point. $\gamma$ is bounded below by the volume entropy, that is,
\[
\gamma\geq \limsup_{R\to\infty}\frac{\log {\rm vol}(B_{x_o}(R))}{R},
\]
 and
\[
\gamma\geq -\limsup_{R\to\infty}\frac{\log ({\rm vol}(M)-{\rm vol}(B_{x_o}(R)))}{R}
\]
if  ${\rm vol}(M)<\infty$.
\end{corl}

\begin{proof}
 By~\eqref{3-1a}, we have
\[
\gamma+\eps\geq \frac{\log {\rm vol}(B_{x_o}(R))}{R}-\frac{\log (C(\eps){\rm vol}(B_{x_o}(1)))}{R}.
\]
The first case follows by  letting $R\to\infty$ and then $\eps\to 0$.

Similarly,
if ${\rm vol}(M)<\infty$, then by ~\eqref{3-1b},
\[
\gamma+\eps\geq -\frac{\log ({\rm vol}(M)-{\rm vol}(B_{x_o}(R)))}{R}+\frac{\log (c(\eps){\rm vol}(B_x(1)))}{R}.
\]
The second case again follows by letting $R\to\infty$ and  then $\eps\to 0$.
\end{proof}

We make the following remark.
\begin{remark}
In Sturm~\cite{sturm}, the concept of uniformly subexponential growth of manifold was introduced. This corresponds to $\gamma=0$. \end{remark}

One of the main technical results in this paper is the following

\begin{lem}\label{lem41} Let $M$ be a complete Riemannian manifold of dimension $n$, with Ricci curvature  and Weitzenb\"ock tensor on $k$-forms bounded below. We define the operator $H= \Delta-c $, where $c$ is a fixed constant  such that $\Delta-c$ is   a nonnegative operator   on $L^2(\Lambda^k(M))$.  For $\xi\in \mathbb{R}$ we denote by ${\vec g}_{m,\xi}(x,y)$ the kernel of the resolvent operator $(H+\xi^2)^{-m}$. Let  $m>n/4.$  Let $\sigma>0$, and define
\[
S(x,y) = \int_0^\infty e^{- \sigma t}|\cos(t\sqrt H)  {\vec g}_{m,\xi}(x,y)| dt.
\]
Denote  by $A_R$  the annulus $\{ x \mid R\leq d(x,y)<R+1\}$ for $ R\geq 0$.

Then  for any $\eps>0$, $y\in M$  and $\xi$ large enough
\begin{equation}\label{est-9}
\left(\int_{A_R} |S(x,y)|^2\, dx \right)^{1/2}\leq C\phi(y) \sigma^{-1}e^{-\left( \sigma-\eps \right)R}.
\end{equation}

Moreover, let $\gamma$ be the exponential rate of volume growth of $M$ and let $\sigma>\gamma/2$. Then
\[
\int_M|S(x,y)| dx\leq C  \sigma^{-1}(
2\sigma-\gamma)^{-1}.
\]
\end{lem}

\begin{proof}
Let  $R'\geq 0$ and denote by  $\rho_{R'}(x)$ the characteristic function of the ball   $B_y(R')$, that is, the value of  $\rho_{R'}(x)$ is 1 on the ball $B_y(R')$  and $0$ otherwise.  We write
\[
{\vec g}_{m,\xi}(x,y)=g_{\xi}^1(x,y)+g_{\xi}^2(x,y)=\rho_{R'}(x)\,  {\vec g}_{m,\xi}(x,y)+(1-\rho_{R'}(x))\, {\vec g}_{m,\xi}(x,y).
\]
Then
\[
|S(x,y)|\leq |S^1(x,y)|+|S^2(x,y)|,
\]
where
\[
S^i(x,y)=\int_0^\infty e^{-\sigma t}|\cos(t\sqrt H) g_{\xi}^i(x,y)| dt
\]
for $i=1,2$.

 We  use the Minkowski inequality to get
\[
\begin{split}
\left( \int_{ A_R} |S^1(x,y)|^2\, dx \right)^{1/2} &
 =    \,
\left( \int_{ A_R} \left(\int_0^\infty e^{-\sigma t}|\cos(t\sqrt H) g^1_{\xi}(x,y)| dt\right)^2 \right)^{1/2} \\
& \leq    \,  \int_{0}^\infty e^{-\sigma t} \left( \int_{A_R}|\cos (t\sqrt{H})g^1_\xi(x,y)|^2 dx \right)^{1/2} dt.
\end{split}
\]
Since the support of $g^1_\xi(x,y)$ is within   $B_y(R')$, by the finite propagation speed of the wave kernel (Theorem~\ref{thmFP}) and the boundedness of the operator $\cos t\sqrt H$  on $L^2$  we obtain
\[
\begin{split}
\int_{0}^\infty & e^{-\sigma t} \left( \int_{A_R} |\cos (t\sqrt{H})g^1_\xi(x,y)|^2 dx \right)^{1/2} dt\\&
= \int_{t>R-R'} e^{-\sigma t} \left( \int_{A_R}  |\cos (t\sqrt{H})g^1_\xi(x,y)|^2  dx \right)^{1/2} dt\\
&= \int_{t>R-R'} e^{-\sigma t}\left( \int_{M}  |\cos (t\sqrt{H})g^1_\xi(x,y)|^2  dx \right)^{1/2} dt\\
&\leq \int_{t>R-R'} e^{-\sigma t}\left( \int_M  |g^1_\xi(x,y)|^2  dx \right)^{1/2} dt.
\end{split}
\]
By estimate \eqref{upper-e} of Corollary \ref{corl32} we have
  \[
 \int_{M}  |g^1_\xi(x,y)|^2  dx \leq  \int_{M}  |\vec g_{m,\xi}(x,y)|^2  dx \leq C_6 \, \phi^2(y)
  \]
 for all $\xi$ large enough. Substituting this in the above estimate gives
\begin{equation}\label{5_e1}
\left( \int_{ A_R} |S^1(x,y)|^2\, dx \right)^{1/2} \leq C\, \phi(y)\sigma ^{-1}e^{-\sigma(R-R')}.
\end{equation}

On the other hand, using the Minkowski inequality once again, we get
\[
 \|S^2(x,y)\|_{L^2}\leq C\sigma^{-1} \, \|g^2_\xi(x,y)\|_{L^2}.
\]
From Corollary \ref{corl32}, estimate~\eqref{upper-e2},  we have
\[
  \int_{M}  |g^2_\xi(x,y)|^2  dx\leq C\phi^2(y) e^{-  C_7 \xi R'}.
 \]

Combining the  above two estimates,  we obtain
\[
\left( \int_{A_R} |S(x,y)|^2\, dx \right)^{1/2}  \leq C \phi(y) \sigma^{-1}(e^{-\sigma(R-R')}+e^{-\frac 12  C_7\xi R'}).
\]

 Fix $\eps>0$ small and  choose $R'=\eps'R$ fixed such that $\sigma\eps'<\eps$. Then whenever $\xi$ is large enough, ~\eqref{est-9} holds.
 This proves the first part of the lemma.

For the second part, we let $\sigma>\gamma/2$ and assume that $\sigma=\gamma/2+2\eps$. By the definition of $\gamma$, we have ${\rm vol}(A_R)\leq C (\eps)\phi(y)^{-2}e^{(\gamma+\eps)R}$.
By the Cauchy-Schwarz inequality
\[
\int_{A_R} |S(x,y)|\, dx \leq \sqrt{\int_{A_R} |S(x,y)|^2\, dx}\cdot \sqrt{{\rm vol}(A_R)}< C  \sigma^{-1}\, e^{-\frac\eps 2R}.
\]
It follows that
\[
\begin{split}
 \int_M|S(x,y)| \, dx    \leq C   \sigma^{-1} \int_0^{\infty} e^{-\frac\eps 2t}dt\leq C  (\sigma \eps)^{-1}.
 \end{split}
\]
This completes the proof.
\end{proof}

 We are now ready to prove  the  main result of this  paper, Theorem \ref{thm3}.
\begin{proof}[Proof of Theorem \ref{thm3}]
 We will first prove the theorem for $p=1$ and then proceed by interpolation for the remaining $p$.

 Let $\xi>0$ be a real number to be chosen big enough later.  For any positive integer $m>n/4$, we have the resolvent identity
\begin{equation} \label{5_e2}
(H-z^2)^{-1}=K
+(\xi^2+z^2)^{-m}(H-z^2)^{-1}(H+\xi^2)^{-m},
\end{equation}
where
\[
K=(H+\xi^2)^{-1}+\cdots+(\xi^2+z^2)^{m-1}(H+\xi^2)^{-m}.
\]

From now on we assume $z=\alpha+i\beta$ with $\beta=\gamma/2+\eps_o$ for some $\eps_o>0$. By  Corollary \ref{lem31},  if $\xi\gg 0$,   the operator
$K$ is bounded on $L^1(\Lambda^k(M))$  for any $m \geq 1$.

Let $S^z(x,y)$ be the kernel of the operator $(H-z^2)^{-1}(H+\xi^2)^{-m}$ and ${\vec g}_{m,\xi}$ as in Lemma \ref{lem41}.   Then
\[
S^z(x,y)=(H-z^2)^{-1} {\vec g}_{m,\xi}(x,y)= (-iz)^{-1}\int_0^\infty e^{izt}\cos(t\sqrt H)  {\vec g}_{m,\xi}(x,y) dt
\]
(see for example \cite{DST}*{Eq (2.3)}), and as a result,
\[
 |S^z(x,y) | \leq  C  \int_0^\infty e^{-(\frac\gamma 2+\eps_o)t}|\cos(t\sqrt H)  {\vec g}_{m,\xi} (x,y)| dt.
\]
By Lemma~\ref{lem41}  with $\sigma =\frac\gamma 2+\eps_o$, we conclude that
\[
\sup_{y\in M}\int_M|S^z(x,y)|dx\leq C<\infty.
\]
Thus the operator $(H-z^2)^{-1}(H+\xi^2)^{-m}$ is bounded on $L^1(\Lambda^k(M))$  whenever ${\rm Im}(z)>\gamma/2$.  Since both operators on the right side of \eqref{5_e2} are bounded on $L^1(\Lambda^k(M))$  we have that  $(H-z^2)^{-1}$ is also bounded on $L^1(\Lambda^k(M))$ whenever ${\rm Im}(z)>\gamma/2$. By replacing $z$ with $-z$ we get the theorem  in the case $p=1$ for $|{\rm Im}(z)|>\gamma/2$.

Since $(H-z^2)^{-1}=(H-\alpha^2+\beta^2-2\alpha\beta i)^{-1}$, whenever $\alpha\neq 0$  the operator has a nonzero imaginary component and is therefore bounded on $L^2$. If on the other hand $\alpha=0$, then $(H-z^2)^{-1}=(H+\beta^2)^{-1}$ which is bounded on $L^2$ if $\beta\neq 0$.    Hence, $(H-z^2)^{-1}$ is a bounded operator on $L^2$ whenever $|{\rm Im}\, z|>0$. Alternatively, replacing $z$ with $iz$, we have that $(H+z^2)^{-1}$ is a bounded operator on $L^2$ whenever $|{\rm Re}\, z| >0$  and that $(H+z^2)^{-1}$ is bounded on $L^1(\Lambda^k(M))$ whenever $|{\rm Re} \,z |>\gamma/2.$

We fix $a >0$ and $b \in \mathbb{R}$ and define the operator
\[
T(x+iy) =  \left(H + \tfrac{\gamma^2}{4}(x +a + iy+ib )^2\,\right)^{-1}.
\]
From the above, it is clear that $T(iy)$  is a bounded operator on $L^2$ and that $T(1+iy)$  is a bounded operator on $L^1$.

We fix $p$ such that $1<p <2$ and let $\lambda\in(0,1)$ be the unique constant with the property $\frac 1p = \lambda + \frac 12 (1-\lambda)$. By the Stein Interpolation Theorem of Lemma \ref{SInt} we get that
\[
T(\lambda) =T( \tfrac 2p -1)=  \left(H + \tfrac{\gamma^2}{4} (\tfrac 2p -1+a +i b )^2\,\right)^{-1}
\]
is bounded on $L^p(\Lambda^k(M))$.    In other words, $(H+z^2)^{-1}$ is bounded in $L^p$ whenever
${\rm Re} \,z = \gamma(\frac 1p -\frac 12+a)$. Since this is true for any $a>0$ and $b\in \mathbb{R}$, we have that $(H+z^2)^{-1}$ is bounded in $L^p(\Lambda^k(M))$ whenever
${\rm Re} \,z >\gamma(\frac 1p -\frac 12)$. By replacing $z$ with $-z$ and $z$ with $-iz$ as before, we conclude that $(H-z^2)^{-1}$ is bounded in $L^p(\Lambda^k(M))$   for any $|{\rm Im} \,z|>(1/p-1/2)\gamma$.

For $p>2$, note that $\left|\frac 1p-\frac 12\right|= \left|\frac{1}{p^*}-\frac 12\right|$ whenever $1-1/p =1/p^*$. Hence the result also holds by the duality of the $L^p$ spaces. This concludes the proof of the theorem.
\end{proof}

Using the above method we will generalize our result to functions of the Laplacian. More general functions were considered by Taylor \cite{tay}, but again, we remove the assumption of bounded geometry for the manifold. We make the following assumptions.
\begin{enumerate}
\item  Let $\gamma_o>0$ be a positive number.  Let $W$ be a horizontal strip in the complex plane $\C$ defined by
\[
W=\{w\mid |{\rm Im}(w)|<\gamma_o/2+\eps_o\}
\]
for some $\eps_o>0$. Let  $f(w)$ be an even  holomorphic function on $W$ satisfying
\[
|f^{(j)}(w)|\leq\frac{C_j}{(1+|w|)^{j}}
\]
for all $0\leq j\leq  n/2+2$ and $w\in W$, where $f^{(j)}$ denotes the derivative of order $j$ of $f$.
\item
Let $g(w)=f(\sqrt w)$ and $c\gg 0$.  Assume that  the inverse Laplace transform $\tilde g(t)$ of $g(s)$ exists for $s\geq c$, that is, for any  real number $s\geq c$, we have
\[
 g(s)=\int_0^\infty e^{-s t} \tilde g(t) dt.
\]
We further assume that  $\tilde g$ is of at most exponential growth,
\[
|\tilde g(t)|\leq c_1 e^{c_2 t} \ \ \text{for \ \ all} \ \   t\in [0,\infty)
\]
for constants $c_1,c_2>0$.

\end{enumerate}
\vspace{.1in}

{\bf Example.} Let $z$ be a fixed complex number such that $|{\rm Im}(z)|>\gamma_o/2+2\eps_o$. Let
\[
f(w)=\frac{1}{w^2-z^2}, \qquad g(w)=\frac{1}{w-z^2}.
\]
Then $f,g$ satisfy the above assumptions. In particular, the inverse Laplace transform of $g$ exists for $w>|z^2|$.

 \begin{thm}\label{thm41-2}
Let $M$ be a manifold with exponential rate of volume growth $\gamma\leq\gamma_o$ and such that its Ricci curvature is   bounded below by $-K_1$  and the Weitzenb\"ock tensor on $k$-forms  is   bounded below by $-K_2$.  We consider functions $f,g$ which satisfy the above assumptions (1), (2), and let $L=\sqrt{\Delta-c}$,  where $c$ is a  constant such that $\Delta-c$ is   a nonnegative operator   on $L^2(\Lambda^k(M))$.
Then $f(L)$ is a bounded operator on $L^1(\Lambda^k(M))$.
   \end{thm}

\vspace{.2in}
For any positive integer $N$, define
  \begin{align}
  &A_N=\sum_{j=0}^{N-1} (-1)^j\frac{\alpha^j}{j!} g^{(j)}(\Delta -c+\alpha);\label{dN-1}\\
  & B_N=\frac{(-1)^N\alpha^{N}}{(N-1)!}\int_0^1 g^{(N)}(\Delta  -c+t\alpha)t^{N-1} dt\label{dN-2}
    \end{align}
for  $\alpha>0$  as in~\eqref{46}. Then $f(L)=g(\Delta-c)=A_N+B_N$ from Appendix ~\ref{appB}. We begin by proving the following three lemmas.

 \begin{lem} Suppose that the Weitzenb\"ock tensor on $k$-forms is bounded below by $-K_2$ over $M$. Then for $\alpha\gg 0$, $A_N$ is a bounded operator on $L^1(\Lambda^k(M))$.
 \end{lem}

 \begin{proof}
 Let $\tilde g(t)$ be the inverse Laplace  transform of $g(s)$. By our assumption,
 \[
 g(s)= \int_0^\infty e^{-st} \tilde  g(t)  \,dt
 \]
 for any real number   $s\gg 0$.
Thus
 \[
 \sum_{j=0}^{N-1}(-1)^j\frac{\alpha^j}{j!}\, g^{(j)}(s+\alpha)= \int_0^\infty e^{-s t}\sum_{j=0}^{N-1} \frac{(\alpha t)^j}{j!} e^{-\alpha t}\, \tilde g(t)  dt.
 \]
 Since $\tilde g(t)$ is of exponential growth at most, there is a constant $C>0$ such that
 \[
 \left| \sum_{j=0}^{N-1}(-1)^j\frac{\alpha^j}{j!}\, g^{(j)}(s+\alpha)\right|\leq C\int_0^\infty e^{-\frac 12\alpha t} e^{-st} dt
 \]
 for $\alpha\gg 0$.
 Let $A_N(x,y)$ be the kernel of $A_N$. Then, by~\eqref{kheat2}  and functional calculus we have
 \[
 |A_N(x,y)|\leq C \int_0^\infty e^{-(\frac 12\alpha-K_2) t} \, h(t,x,y) dt,
 \]
where $h(t,x,y)$ is the heat kernel on functions.  As in the proof of Corollary \ref{lem31}, since $\int_M h(t,x,y) dx\leq 1$, we  have
\[
\int_M|A_N(x,y)| dx\leq C<\infty
\]
and hence $A_N$ is $L^1$ bounded whenever $\alpha$ is large enough.
\end{proof}

 \begin{lem}
 For any $m<N$, the operator $R=R_N=(\Delta   -c  +\alpha)^mB_N$ is bounded on $L^2(\Lambda^k(M))$.
 \end{lem}
 \begin{proof}
 Let
 \[
 \tilde \Delta_t=\Delta-c +t\alpha.
 \]
 Then we can rewrite $R$ as
 \[
 \begin{split}
 &
 R=\frac{(-1)^N\alpha^{N}}{(N-1)!}\int_0^1(\Delta   -c  +\alpha)^m g^{(N)}(\Delta  -c+t\alpha)t^{N-1} dt\\
 &=
 \frac{(-1)^N\alpha^{N}}{(N-1)!}\int_0^1 (\tilde\Delta_t+(1-t)\alpha)^m g^{(N)}(\tilde\Delta_t)t^{N-1} dt.
 \end{split}
 \]
 Using Minkowski's inequality,  it suffices to show that for any $t\in[0,1]$, the operator
 $
(\tilde\Delta_t+(1-t)\alpha)^m g^{(N)}(\tilde\Delta_t)
 $
is uniformly bounded on $L^2$  with respect to $t$. This can easily be seen by the Spectral Theorem since the functions $(x+(1-t)\alpha)^m g^{(N)}(x)$ are  uniformly bounded by Lemma \ref{blem40} whenever $m<N$.

 \end{proof}

 \begin{lem}
Suppose that the Ricci curvature of $M$ is bounded below by $-K_1$, and the Weitzenb\"ock tensor on $k$-forms is bounded below by $-K_2$ over $M$. Assume that the  exponential rate of volume growth $\gamma$ of $M$ satisfies $\gamma\leq\gamma_o$.
Then $B=B_N$ is bounded on $L^1(\Lambda^k(M))$.
 \end{lem}

\begin{proof}
 Let $m>n/4$. For any $s\in[0,1]$, we define
\[
\varphi_s(t)=g^{(N)}(t^2)\, (t^2+(1-s)\alpha)^m.
\]
By definition $\varphi_s(x)$ and its Fourier tranformation $\hat\varphi_s(\xi)$ satisfy the relations
\[
\begin{split}
&\hat\varphi_s(\xi)=\int_\R\varphi_s(t) e^{-i t\xi} dt,\\
&\varphi_s(t)=(2\pi)^{-1}\int_\R\hat\varphi_s(\xi) e^{it\xi} d\xi.
\end{split}
\]
Since $\varphi_s(x)$ is an even function, the above two equations are reduced to
\[
\begin{split}
&\hat\varphi_s(\xi)=2\int_0^\infty \varphi_s(t) \cos(\xi t) dt,\\
&\varphi_s(t)=\pi^{-1}\int_0^\infty\hat\varphi_s(\xi) \cos(t\xi)d\xi.
\end{split}
\]
Using the definition of $\varphi_s(x)$, we can rewrite $B_N$ as
\[
B_N=\frac{(-1)^N\alpha^{N}}{(N-1)!}\left(\int_0^1 \varphi_s(\sqrt{\Delta-c+s\alpha})s^{N-1} ds\right) (\Delta-c+\alpha)^{-m}.
\]

In order to prove that $B_N$ is bounded on $L^1(\Lambda^k(M))$, it suffices to show that for any $s\in[0,1]$, the operators
\[
\varphi_s(\sqrt{\Delta-c+s\alpha}) (\Delta-c+\alpha)^{-m}
\]
are uniformly bounded on $L^1(\Lambda^k(M))$.

Let $r(x,y)=r_s(x,y)$ be the kernels of the above operators. Then
\[
\begin{split}
r(x,y)&=\varphi_s(\sqrt{\Delta-c+s\alpha}) \vec g_{m,c,\alpha}(x,y)\\
&= \pi^{-1}\int_0^\infty\hat \varphi_s(\xi) \cos \left( \xi \sqrt{\Delta   -c +s\alpha} \right)\,\,\vec g_{m,c,\alpha}(x,y)\,d\xi,
\end{split}
\]
where $\vec g_{m,c,\alpha}(x,y)$ is the kernel of $(\Delta-c+\alpha)^{-m}$.

By the assumption on $f$, the function $\varphi_s(w) =g^{(N)}(w^2)\, (w^2+(1-s)\alpha)^m$  satisfies $|\varphi_s(w)| \leq C \, (1+|w|)^{ 2m-N }$ in the strip $|{\rm Im}\, w |<\gamma_o/2+3\eps_o/4$ by Lemma~\ref{blem40}.  Choosing $m$ such that $n/2 < 2m < N +1\leq n/2+2$ we get $|\varphi_s(w)| \leq C \, (1+|w|)^{1+\beta}$ for some $\beta>0$. Therefore, $\varphi_s(w)$ satisfies the $L^1$-integrablity criterion of Lemma ~\ref{blem41} over $x$ when $w=x+i\tau$ for $|\tau|\leq \eps_o/2$. By the same lemma
\[
|\hat\varphi_s(\xi)|\leq C e^{-(\gamma_o+\eps_o)\xi/2}.
\]
As a result,
\[
 |r(x,y) | \leq  C  \int_0^\infty e^{-(\frac{\gamma_o} 2+\frac{\eps_o}{2}) \xi}|\cos(\xi\sqrt{\Delta  -c  +s\alpha})\,\, \vec g_{m,c,\alpha}(x,y)| d\xi.
\]
If we replace $c$ by $c-s\alpha$, and $\xi$ by $\sqrt{ (1-s)\alpha}$ in Lemma~\ref{lem41},  then
\[
\int_M|r(x,y)| dx\leq C
\]
uniform for $y$ and $s\in[0,1]$. This completes the proof of the lemma.
\end{proof}

\begin{proof}[Proof of Theorem~\ref{thm41-2}] As we have seen,
\[
f(L)= g(\Delta -c)=A_N+B_N,
\]
with $A_N, B_N$ defined by \eqref{dN-1} and \eqref{dN-2}.  By the above three lemmas $f(L)$ is bounded in $L^1(\Lambda^k(M))$ for sufficiently large $\alpha>0$.

\end{proof}

\begin{corl}
We consider functions $f,g$ which satisfy the above assumptions (1), (2).
Let $M$ be a manifold with exponential rate of volume growth $\gamma$ satisfying
\[
\left|\frac 1p-\frac 12\right|\gamma \leq\frac{\gamma_o}{2},
\]
as well as Ricci curvature and Weitzenb\"ock tensor bounded below.  Let $L=\sqrt{\Delta-c}$,  where $c$ is a  constant such that $\Delta-c$ is   a nonnegative operator   on $L^2(\Lambda^k(M))$. Then $f(L)$ is a bounded operator on $L^p(\Lambda^k(M))$.
\end{corl}

\begin{proof}
This follows from the Stein  Intepolarion Theorem (Lemma~\ref{SInt}) as in Taylor \cite{tay}.
\end{proof}

\begin{remark}
If $\gamma=0$, our result is reduced to the result of Sturm~\cite{sturm}. The main difference is that in the proof of Sturm's result, only
heat kernel estimates are needed, whereas in our proof, we need to use both the heat kernel estimates and the wave kernel. The use of the wave kernel on manifolds that might not have bounded geometry seems new, and we expect that this method will have further applications in geometric analysis.
\end{remark}
\section{The $L^p$-spectrum of hyperbolic space} \label{S3}

In this and the following section, we provide a comprehensive study of the spectrum of the Hodge Laplacian on $L^p(\Lambda^k(\mathbb{H}^{N+1}))$, where $\mathbb{H}^{N+1}$ is the $N+1$-dimensional hyperbolic space. We would like to point out that for $p=2$, or $k=0$ (the Laplacian on functions), our results are partially known.

Since  $\mathbb{H}^{N+1}$ is a homogeneous manifold, the $L^2$ spectra and the essential spectra coincide, that is  $\sigma(2,k)=\sigma_{\rm ess}(2,k)$, and we will see that this is also true for the $L^p$ spectrum. Our main goal in this section will be to prove the first part of Theorem \ref{thm1} by computing $\sigma(p,k)$. We begin with the following observation regarding the parabola $Q_{p,k}$.
\begin{lem} \label{lemQp}
For $1\leq p \leq 2$ the   parabolic region $Q_{p,k}$   has as its boundary the parabola
\begin{equation}\label{Qp}
P_{p,k} = \left\{-{ \left(\tfrac Np + is -k\right)\left( N \left(\tfrac 1p-1\right) + is +k\right)}\;\Big|\, s\in \mathbb{R}\, \right\}
\end{equation}
 and is the region to the right of $P_{p,k}$.

Moreover, for $1\leq p\leq 2$, we have
\begin{equation}\label{containment}
Q_{p,k}=\bigcup_{p\leq q\leq 2} P_{q,k}.
\end{equation}
\end{lem}
 Observe that for a fixed $p$ the region   $Q_{p,k}$ increases in $k$ for $1\leq k \leq   (N+1)/2$.  The parabolic region reduces to the set $[(\frac N2-k)^2,\infty)$ in the case $p=2.$ We leave the proof of the Lemma as an exercise to the reader.

To prove the first part of Theorem \ref{thm1}  we will first need to construct a class of approximate eigenforms for every point in the $L^p$-spectrum. We begin by providing a set of smooth $k$-eigenforms which form a global basis of $k$-forms over $\mathbb{H}^{N+1}$.

Consider the upper-half plane model for $\mathbb{H}^{N+1}$
\[
\{(x,y)\,|\, x\in \mathbb{R}^N, y>0\}
\]
endowed with the metric
\begin{equation}\label{rmetric}
g=\frac{1}{y^2}(dx^2+dy^2),
\end{equation}
where $dx^2$ is the Euclidean metric on $\mathbb R^N$.
The Laplacian on smooth functions over $\mathbb{H}^{N+1}$ is given by
\[
\Delta f = - y^{N+1} \frac{\p}{\p y}\left(y^{1-N}\frac{\p f}{\p y}\right) + y^2\, \Delta_x f
\]
where $\Delta_x = -\sum_i \p^2 / {\p x_i}^2$ is the Laplacian over the Euclidian space $\mathbb{R}^N$.

It is clear that the functions $y^\mu$ for any $\mu\in \mathbb{C}$ are formal eigenfunctions of the Laplacian  and satisfy
\[
\Delta y^\mu = -\mu(\mu-N) y^\mu.
\]

If we denote as $dy, dx^1, \ldots, dx^N$  the dual 1-forms to the coordinate frame field ${\p y}, {\p x_1}, \ldots,  {\p x_N},$ then the space of $k$-forms over hyperbolic space has as a basis $k$-forms of the following two types
\[
dy\wedge dx^I \quad {\rm and} \quad dx^J
\]
where $dx^I= dx^{i_1}\wedge \cdots \wedge dx^{i_{k-1}}$ with $i_m \in \{1,\ldots, N\}$ and $i_1<\cdots <i_{k-1}$, whereas  $dx^J= dx^{j_1}\wedge \cdots \wedge dx^{j_{k}}$ with  $j_m \in \{1,\ldots, N\}$ and $j_1<\cdots <j_{k}$.

We observe that
\[
|dy|=|dx^i|= y,\ \ {\rm for  \ \  all}  \ \ i, \ \ \rm{and}
\]
\[
|dy\wedge dx^I|=|dx^J|= y^k
\]
where $|\cdot|$ is the norm induced by the Riemannian inner product.

If  $D$ is the Levi-Civita connection, a standard calculation gives
\begin{align*}
\begin{split}
D_{\p x_i} (dy) = -\frac{1}{y}dx^i,  \qquad &  D_{\p y} (dy) =  \frac{1}{y}dy; \\
D_{\p x_i} (dx^j) = \delta_{ij}\frac{1}{y}dy, \qquad &   D_{\p y} (dx^j) = \frac{1}{y}dx^j,
\end{split}
\end{align*}
and
\begin{align}\label{C1}
\begin{split}
D_{\p x_i} (dy\wedge dx^I) = - \frac{1}{y} dx^i\wedge dx^I,  \qquad  &  D_{\p y} (dy\wedge dx^I) =  \frac{k}{y}
dy\wedge dx^I; \\
D_{\p x_i} (dx^J) = -  \frac{1}{y} \iota(\p x_i) dy \wedge dx^J, \qquad  &   D_{\p y} (dx^J) = \frac{k}{y}dx^J,
\end{split}
\end{align}
where $\iota(\p x_i)$  denotes the contraction operator in the direction $\p x_i.$
Furthermore, for any complex number $\mu$, we have
\begin{align}\label{C3}
\begin{split}
&\delta (y^\mu dy\wedge dx^I )=-(\mu-N-1+2k) y^{\mu+1} \, dx^I\quad  \rm and\\
&\delta  (y^\mu dx^J)=0.
\end{split}
\end{align}

We thus  obtain
\begin{align} \label{D1}
\Delta  (y^\mu \, dy\wedge dx^I) &= d\delta (y^\mu \,dy\wedge dx^I) = -(\mu+1)(\mu-N-1+2k) \, y^\mu\, dy\wedge dx^I \\
\Delta  (y^\mu \, dx^J) &=  \delta  d (y^\mu \,dx^J)  = -\mu (\mu-N+2k)\,y^\mu \,dx^J. \label{D2}
\end{align}

It follows that $y^\mu\, dy\wedge dx^I, \ \ y^\mu\, dx^J$ are formal eigenforms of the Laplacian for any $\mu\in \mathbb{C}$.

We will first prove that the $L^p$-spectrum contains the parabolic region.
\begin{prop} \label{HypThm1}
For $0\leq k \leq (N+1)/2$ and  $1\leq p \leq 2$ the $L^p$-spectrum of the $k$-form Laplacian on  $\mathbb{H}^{N+1}$, $\sigma (p,k)$, contains $Q_{p,k}$ given in \eref{m-1-1}.

In the case when $N$ is odd and $k=(N+1)/2$,  the $L^p$-spectrum also contains the point $\{0\}$ for all $1\leq p \leq \infty$, which is in fact a point in the essential spectrum.
\end{prop}

\begin{proof}
We first observe  that we only need to prove that for  $1\leq p\leq 2$, we have $\sigma(p,k)\supset P_{p,k}$. This is because   for any  $1\leq p\leq q\leq 2$, by the monotonicity of the spectra,  we have $\sigma(p,k)\supset\sigma(q,k)\supset P_{q,k}$. Thus, if $\sigma(q,k)\supset P_{q,k}$, by ~\eqref{containment}
we have
\[
\sigma(p,k)\supset \bigcup_{p\leq q\leq 2} P_{q,k}=Q_{p,k}.
\]

Let $\omega=y^{\mu} \, f(x,y) \, \eta$, where $\eta$ is either $dy\wedge dx^I$ or $dx^J$, $f(x,y)$ is a smooth compactly supported function and $t,s \in \mathbb{R}$.  For $\mu=t+is$ and given that the volume element is $dv= y^{-N-1} dy dx$, we have
\begin{equation} \label{D0}
\|\omega\|_p^p=\int_{\mathbb H^{N+1}} y^{tp + kp-N-1} \,|f(x,y)|^p\; dy dx.
\end{equation}
Setting $\mu=N/p -k +is$ in \eref{D1} and \eref{D2}, the corresponding eigenvalues are
\begin{align} \label{D3}
&  -(N/p -k+1+is)(N/p -N+k-1+is) \\
& -(N/p -k+is)(N/p -N+k+is)  \label{D4}
\end{align}
respectively, for any $s\in \mathbb{R}$. Both of these represent parabolas in the complex plane, and \eqref{D4} is the optimal one for $k$ and $p$ as in the theorem, coinciding with the boundary of $Q_{p,k}$, which is the parabola  $P_{p,k}$ in \eqref{Qp} from Lemma \ref{lemQp}.

Therefore, we consider the approximate eigenforms
\[
\omega_n = f_n(x,y)\, y^{N/p -k+ is } \, dx^J
\]
with
\[
f_n(x,y)=  c_n(\log y) \, b(x)
\]
where $b(x)$ is an appropriate function with compact support in the $x$ variable and  $c_n$ is a function with support in $[-n^{3p}, \log n]$ which takes the value 1 on $[-n^{3p}+1, \log n -1]$ and such that $|c_n'|, |c_n''|\leq C.$ It follows that
\[
\|\omega_n\|_{L^p}^p = \int_{\mathbb{R}^N} \int_{e^{-n^{3p}}}^n y^{-1}\, c_n^p(\log y) \, b^p(x) \,dy dx \geq C' \int_{e^{-n^{3p}+1}}^{n/e} y^{-1} \, dy \geq C'\, n^{3p}.
\]

At the same time, for any smooth function  $f$ and $k$-form $\eta$ we have the identity
\[
\Delta (f\, \eta) = f\, \Delta \eta- 2 \n_{\!\n\! f}\, \eta + (\Delta f) \,\eta.
\]
Using \eref{D2} and the definition of the Laplacian on $\mathbb{H}^{N+1}$ we compute,
\begin{align*}
\Delta \omega_n  = & f_n \Delta (y^{\mu} \,dx^J) + (\Delta f_n\, ) \,y^{\mu} \,dx^J -2 \n_{\n f_n} (y^{\mu} \,dx^J) \\
= &-{ \mu (\mu+2k-N)}\,\omega_n  +  y^{\mu+2}\, c_n \, (\Delta_x b) \,dx^J -(2\mu-N+2k) \, y^{\mu} \, c_n' \,b\,dx^J \\
& \quad -   y^{\mu}  \, c_n''\,b \,dx^J  +2\sum_i  y^{\mu+1}  \frac{\p b}{ \p x_i}\, c_n \, \iota (\p x_i)(dy\wedge dx^J),
\end{align*}
where $\mu=N/p -k+ is $.  Taking $\lambda = -\mu (\mu+2k-N)$ we see that
\begin{align*}
\|\Delta \omega_n -\lambda\omega_n\|_{L^p}^{p} \leq & C\int_{\R^n} \int_{{\rm spt}\, c_n'} \, y^{-1} \,|b|\,dy +  C \int_{\mathbb{R}^n} \int_{{\rm spt}\, c_n} \, y^{2p-1}\, |\Delta_x b |\,dydx\\
&  + C \int_{\mathbb{R}^N} \int_{{\rm spt}\, c_n} \, y^{p-1} \, |\n_x b |\,dydx.
\end{align*}
The first integral on the right side is bounded. The remaining two integrals are of order $n^{2p}$ and $n^{p}$, respectively.

Therefore, for any $\eps>0$ we can find $n$ sufficiently large such that
\[
\|\Delta \omega_n -\lambda\omega_n\|_{L^p} \leq \eps \|\omega_n\|_{L^p}.
\]
We then conclude that $\lambda$ is in  $\sigma (p,k)$  for $1\leq p \leq 2$.

We will now show that $\{0\}$ is always a point in the essential spectrum for $k=(N+1)/2$, when $N$ is odd. Note that $dy\wedge dx^I$ and $*(dy\wedge dx^I) = dx^{\bar{I}}$ are $*$-dual forms which are both closed and co-closed by \eref{C3}.
Moreover, for any $\nu>0$ the  $(N+1)/2$-form
\begin{equation} \label{halfEF}
\phi = e^{-\sqrt{\nu} \, y } e^{ i\sqrt{\nu} \, x_j } \, dx^j \wedge dx^I + i \, e^{-\sqrt{\nu} \, y } e^{ i\sqrt{\nu} \, x_j } \, d y \wedge dx^I
\end{equation}
is also formally harmonic as it satisfies $\phi =\delta \phi= \Delta \phi =0 $
whenever $dx^I$ is an $(N-1)/2$ - form such that $dx^j \wedge dx^I \neq 0$. Let $\eps>0$. By multiplying $\phi$ with appropriate cut-off functions $b_\eps=b_\eps(x)$ which is  compactly supported   in $\mathbb{R}^N$, we can get
\[
\|\Delta(b_\eps \phi)\|_{L^2} \leq C\, \left\| \; \left(|y^2 \Delta b_\eps| +  |\nabla b_\eps| \right)  |\phi|\right\|_{L^2} \leq \eps \| b_\eps \phi \|_{L^2}
\]
since $\phi$ decreases exponentially in $y$ and $\int_0^\infty |\phi|^2 \, y^{-N-1} dy <\infty$.   As a result  $\{0\}$  belongs to $\sigma (2,(N+1)/2)$  (and in fact to $\sigma_\textup{ess}(2,(N+1)/2)$).

Similarly, given the exponential decay of $\phi$ in $y$, and the fact that $\{0\}$ belongs to the $L^1$-spectrum of the Laplacian on $\mathbb{R}^N$, we can also find cut-off functions $b_\eps=b_\eps(x)$ which are compactly supported  in $\mathbb{R}^N$, such that
\[
\|\Delta(b_\eps \phi)\|_{L^1} \leq C\, \left\| \; \left(|y^2 \Delta b_\eps| +  |\nabla b_\eps| \right)  |\phi|\right\|_{L^1} \leq \eps \| b_\eps \phi \|_{L^1}
\]
Therefore, $\{0\}$  belongs to $\sigma (1,(N+1)/2)$ and  in consequence it also belongs to  $\sigma (p,(N+1)/2)$ for all $p \in [1,\infty]$.
\end{proof}

 By the remark after Theorem \ref{thm62} the above argument is only required for $\frac{2N}{N+1}<p<\frac{2N}{N-1}$.

We observe that in the case $p=2$ the parabola $P_{2,k}$ collapses to an interval in the nonnegative real line. The minimum value of the coefficient   in \eref{D3} is
\begin{equation} \label{lam1}
\lambda_1=\left(\tfrac N2-k+1\right)^2,
\end{equation}
and the corresponding minimum value of the coefficient in \eref{D4} is
\begin{equation} \label{lam2}
\lambda_2=\left(\tfrac N2-k\right)^2.
\end{equation}
As Donnelly proved in \cite{Don2}, $\lambda_2$ is the bottom of $\sigma (2,k)$ for $0\leq k \leq N/2$, and of the spectral gap for $k=(N+1)/2$. In other words, our parabolic region coincides with the $L^2$-spectrum in this case.

\begin{prop}[Donnelly \cite{Don2}] \label{HypThm2}
The $L^2$-spectrum of the Laplacian on $k$-forms over the hyperbolic space $\mathbb{H}^{N+1}$ is given by
\[
\sigma (2,k) = \sigma  (2,N+1-k) =\left[\,\left(\tfrac{N}{2} - k\right)^2, \infty\,\right)
\]
for $0\leq k \leq N/2$, and whenever $N$ is odd
\[
\sigma (2,(N+1)/2)= \sigma_{\mathrm ess} (2,(N+1)/2) =  \{0\} \cup \left[\,\tfrac 14, \infty\,\right).
\]
\end{prop}

In fact, Mazzeo and Phillips prove that the $L^2$ essential spectrum on $k$-forms is the same as the one above over quotients of hyperbolic space $\mathbb{H}^{N+1}/\Gamma$ that are geometrically finite and have infinite volume \cite{mazz}.  In these case, isolated eigenvalues can also appear in the $L^2$ spectrum.

To prove that the $L^p$-spectrum  $\sigma (p,k)$ is exactly the set of points in the parabolic region $Q_{p,k}$ for $0\leq k \leq N/2$, it suffices to show  that $(\Delta -\xi)^{-1}$ is a bounded operator on $L^p$ for the points $\xi \in  Q_{p,k}$. We will first compute the $L^1$-spectrum and then via interpolation find the $L^p$-spectrum using the known $L^2$-spectrum.  The case $k=(N+1)/2$ for $N$ odd needs to be treated separately as it requires a more delicate argument to prove the existence of a gap.

Since the hyperbolic space $\mathbb{H}^{N+1}$ is symmetric and has uniform exponential growth rate $\gamma=N$, Theorem \ref{thm3} immediately gives the following result.
\begin{lem} \label{lem7}
Let $0\leq k \leq \frac{N}{2}$. Whenever $|{\rm Im}\, z|>\frac N2$,  the resolvent operator $(\Delta -(\frac N2 -k)^2 -z^2)^{-1}$ is a bounded operator on $L^1$.
\end{lem}
Here we have also used the value for $\lambda_1$ as given by Proposition \ref{HypThm2}.  By Proposition~\ref{HypThm1} and Lemma \ref{lem7} we immediately get the $L^1$-spectrum.
\begin{corl}
For $0\leq k \leq N/2$, \ $\sigma(1,k)  $ is exactly  the closed parabolic region  $Q_{1,k}$. For $k>(N+1)/2$ the $L^1$-spectrum is given by duality as in Theorem \ref{thm1}.
\end{corl}

The $L^p$-spectrum is now found using interpolation.

\begin{prop} \label{thm41}
 $\sigma(p,k)$ for the Laplacian on $\mathbb{H}^{N+1}$ is exactly as in Theorem \ref{thm1}.
\end{prop}

\begin{proof}
Let $0\leq k \leq N/2$.  We consider the operator $H=\Delta - \lambda_1$ with $\lambda_1=(\frac N2-k)^2$ since $\sigma (2,k)=[(\frac N2-k)^2,\infty).$ Hence, $(H-z^2)^{-1}$ is a bounded operator on $L^2$  for $|{\rm Im}\, z|>0$, and by Lemma \ref{lem7} it is bounded on $L^1$ for $|{\rm Im} \,z|>N/2.$ The same argument using Stein interpolation as in the proof of Theorem \ref{thm41}, with $\gamma=N$, gives us that $(H-z^2)^{-1}=(\Delta-(\frac N2-k)^2-z^2)^{-1}$ is bounded on $L^p$ whenever $|{\rm Im}\, z|> N\, \left(\frac 1p - \frac 12\right)$. It is in other words bounded on the complement of the region $Q_{p,k}$. Using Proposition \ref{HypThm1} we get that $\sigma(p,k)=Q_{p,k}$ for $0\leq k \leq N/2$. The case for $k>(N+1)/2$ follows by Poincar\'e duality.
 \end{proof}

We will now address the case of  $N$ odd and $k = (N+1)/2$. This is the more subtle of the cases and requires the analytical arguments or Section \ref{S0} to prove the existence of an isolated point in the $L^p$-spectrum. Recall from Proposition \ref{HypThm1} that  $\{0\}$ belongs to the essential spectrum  $\sigma_\textup{ess}(p,(N+1)/2)$ . On the other hand, since $\{0\}$ is a  point  outside  the parabola $Q_{p, (N+1)/2} $  {for $\frac{2N}{N+1}<p<\frac{2N}{N-1}$,  Stein Interpolation cannot be used in the computation of the spectrum. Instead, our argument relies on  Proposition \ref{kspec}  and observation \eqref{Qhalf-2} below.
\begin{thm} \label{thm51}
Let $1\leq p \leq 2$.
In the case of $N$ odd,   the $L^p$-spectrum (and in consequence the essential spectrum) of the Laplacian on $(N+1)/2$-forms over the hyperbolic space $\mathbb{H}^{N+1}$  is ${\displaystyle \sigma(p,(N+1)/2) = Q_{p,(N+1)/2} \cup \{0\} }$.

For $p\geq 2$  the spectrum is given by duality as in Theorem \ref{thm1}.
\end{thm}
\begin{proof}
We saw that $\{0\}$ belongs to $\sigma_\textup{ess}(p,(N+1)/2)$ in  Proposition \ref{HypThm1}.  We now make the following key observation. In the case that $N$ is odd and for any $p$
\begin{equation}\label{Qhalf-2}
P_{p, (N+1)/2} = P_{p, (N-1)/2} = P_{p, (N+3)/2} = \left\{- \left(\tfrac Np -\tfrac N2 + is\right)^2  + \frac 14 \;\Big|\, s\in \mathbb{R}\, \right\}.
\end{equation}
(Note that all three parabolas cross the $x$-axis at $1/4$ when $p=2$.)

 Moreover, since the orders $(N-1)/2$ and $(N+3)/2$ for forms are dual in this case, then  $\sigma (p,(N-1)/2)=\sigma (p,(N+3)/2)$, and these spectra consist of all points to the right of the parabola  $P_{p, (N-1)/2} = P_{p, (N+3)/2}$   by what we have proved so far.

By applying Proposition \ref{kspec} for the case $k=(N+1)/2$, and using the fact that $\sigma  (p,(N-1)/2) = \sigma (p,(N+3)/2)$,  we get
\[
\sigma(p,(N+1)/2) \setminus\{0\} \subset \sigma (p,(N-1)/2)
\]
for all $p$. By Proposition \ref{HypThm1} and \eref{Qhalf-2}, the boundary of the set in the right side is contained in the set on the left side, and the two sets must be equal.
\end{proof}

\begin{proof}[Proof of first part of Theorem \ref{thm1}]
The first part of Theorem  \ref{thm1} now follows from Propositions \ref{HypThm1}, \ref{thm41} and Theorem \ref{thm51}.
\end{proof}

\section{Eigenvalues in the $L^p$-spectrum of $\mathbb{H}^{N+1}$} \label{S7}

Our main goal in this section is to prove that any point in the $L^p$-spectrum of the $k$-form Laplacian over $\mathbb{H}^{N+1}$ is in fact an eigenvalue if and only if $p>2$,  and to study the existence of harmonic $k$-forms when $k$ is half the dimension of the manifold.
We use the spherically symmetric model for the hyperbolic space $\mathbb H^{N+1}=[0,\infty)\times S^{N}$ with metric
\[
ds^2=dr^2+ f^2(r) d\omega^2,\qquad \text{with} \ \ f(r)=\sinh r
\]
and where $d\omega^2$ is the standard metric on the unit sphere $S^N$.

Using equation (1.3) from Donnelly \cite{Don2}  we can simplify the action of the Laplacian on certain types of forms.
\begin{lem} Let $\eta$ be a $k$-form over $S^{N}$. For any smooth $k$-form of the type $\omega=\phi(r)\,\eta$ we have  
\begin{equation*}
\begin{split}
\Delta \omega
&= -(\phi''  + (N-2k) \phi' \, f' \, f^{-1})\,\eta  -2 \phi \, f' \,f^{-3} \, dr \wedge  \delta_S \eta +  \phi f^{-2} \, \Delta_S\eta
\end{split}
\end{equation*}
 where $\delta_S$ and $\Delta_S$ are the adjoint to the exterior derivative and the Laplacian on $S^N$, respectively.  \end{lem}

We now assume that the $k$-form $\eta$ is a co-closed eigenform on the sphere such that $\delta_S\eta_1=0$ and $\Delta_S\eta=\lambda\eta$.  Then from the  above lemma, we get
\[
\Delta (\phi  \eta) = \Delta_1(\phi )  \eta,
\]
where $\Delta_1$ is the operator on functions in one variable given by
\begin{equation}\label{S7eq2}
\Delta_1 \phi= -(\phi''  + (N-2k) \phi' \, f' \, f^{-1})   + \lambda \,  f^{-2}\,\phi \,.
\end{equation}


We will begin by completing the proof of Theorem \ref{thm1}.

\begin{proof}[Proof of second part of Theorem \ref{thm1}]
Let $\Lambda=x+iy\in Q_{p,k}$. We try to find a suitable (complex-valued) function $\phi=\phi(r)$ such that $\phi\eta$ is an eigenform corresponding to the eigenvalue $\Lambda$. By  \eqref{S7eq2} we need to solve the equation.
\begin{equation}\label{ode-2}
\phi''  + (N-2k) \phi' \, f' \, f^{-1}   - \lambda \,  f^{-2}\,\phi +\Lambda\phi=0.
\end{equation}

The above equation has a regular-singular point at $r=0$.  To study the behavior of the solution near $r=0$ we first observe that $f\sim r$ and $f'\sim 1$ for $r$ small, hence for these values the characteristic equation  corresponding to \eqref{ode-2} is given by
\[
\alpha(\alpha-1)+(N-2k)\alpha -\lambda=0.
\]
 This equation has a solution of order $\phi\sim r^\alpha$ with
\[
\alpha=\left(-(N-2k-1)+\sqrt{(N-2k-1)^2+4\lambda}\right)/2.
\]

At the same time, by  the results of Gallot and Meyer \cite{GalMey75} we know that   the eigenvalues of $S^N$ for closed $k$-forms  are given by $\lambda^c=(s+k)(s+N-k+1)$ where $s$ is any nonnegative integer. Hence by duality we get that all eigenvalues for co-closed  $k$-forms are of the type
\[
\lambda = (N-k +s)(s+k+1) \quad \text{with} \quad s\geq 0.
\]
A simple computation for the above $\lambda$ yields $\alpha=k+1+s\geq k+1$. Thus $\phi \eta$ is bounded near the origin. By elliptic regularity, $\phi \eta$  is therefore smooth.

As $r\to\infty$, Equation~\eqref{ode-2} is asymptotic to
\[
\tilde \phi''+(N-2k)\tilde \phi'+\Lambda\tilde \phi=0.
\]
 Let $m=(N-2k)/2$. The above differential equation has two linearly independent solutions, $e^{-m
\pm \lambda_o r}$ where
 \begin{equation}\label{S7eq7}
\begin{split}
&
\lambda_o = \sqrt{m^2 -x-iy} = a + i b\quad \text{with} \,\,
 a>0, \quad \text{and}\\
 & a^2= \frac 12 \left[m^2-x + \sqrt{(m^2-x)^2+y^2} \right].
 \end{split}
\end{equation}
By \cite{Har}*{Theorem 9.1,  p.379}  and the remark that addresses the complex case we get that $\phi$  has growth of order  $e^{ (-m+a)r}$,  at most.

To check that  $\phi(r) \eta $ is $L^p$ integrable it therefore suffices to check its integrability at infinity:
\[
\|h \eta\|_p^p \leq C \int_1^\infty |h\eta|^p f^{N} dr  \leq C \int_1^\infty e^{(-m+a)p r} f^{-kp+N} dr.
\]
For the right side to be bounded it suffices to have
\begin{align} \label{S7eq6}
  ap& -mp-kp+N<0 \Leftrightarrow a<  N  \left(\tfrac 12 -\tfrac{1}{p}\right) \notag \\
  &\Leftrightarrow  m^2-x + \sqrt{(m^2-x)^2+y^2}  <   2 N^2 \left(\tfrac 12 -\tfrac 1p\right)^2 \ \ \text{and} \ \   p>2 \notag \\
  &\Leftrightarrow y^2  < 4 \,\left[x+ N^2 \left(\tfrac 12 -\tfrac 1p\right)^2 - (\tfrac N2 - k)^2\right] N^2(\tfrac 12-\tfrac 1p)^2
\end{align}
after substituting for $m$. It is a straightforward exercise to show that all points $x+iy$ in the complex plane that lie in  the parabolic region $Q_{p,k}$   are precisely the set of points that satisfy inequality \eref{S7eq6}.

For $1\leq p \leq 2$, suppose that $\lambda$ is an eigenvalue in the $L^p$ spectrum with corresponding eigenform $\omega$. By \cite{Char2}*{Theorem 4.2} the heat operator of the Laplacian on $k$-forms is ultracontractive over the hyperbolic space, since the volume of balls of radius 1 is uniformly bounded below and it has constant curvature. By duality the heat operator is also bounded from $L^1$ to $L^2$, and using interpolation it is also bounded from $L^p$ to $L^q$ for any $1\leq p\leq q\leq \infty$.  Since $e^{-t\Delta} \omega = e^{-t\lambda} \omega$, it then follows that $\omega \in L^q$ for all $q\geq 2$, and in particular $\lambda$ must also be an $L^2$-eigenvalue. Since Donnelly \cite{Don2}  has proved that other than $\{0\}$ there are no other eigenvalues in the $L^2$ spectrum, we have a contradiction.
\end{proof}

Note that a similar result was proved by Taylor for the Laplacian on functions over symmetric spaces \cite{tay}.

 Our final goal is to prove Theorem \ref{thm62} about the existence of harmonic $(N+1)/2$-forms when $N$ is odd.


\begin{proof}[Proof of Theorem \ref{thm62}]

We let  $\{(\lambda_k, \eta_k)\}$ be a pair of eigenvalue and   co-closed $(N-1)/2$-eigenform of  $S^N$  with respect to the standard metric of the sphere, such that $\{\eta_k\}$ is a complete orthonormal basis  for the $L^2$-integrable co-closed $(N-1)/2$-eigenforms of  $S^N$. Note that $\lambda_k>0$ for all $k$. Donnelly~\cite{Don2} proved that any  harmonic $L^2$ form of order $(N+1)/{2}$ must be of the type
\begin{equation}\label{eta1}
\omega=\sum_{k=1}^\infty a_k \omega_k+\sum_{k=1}^\infty b_k \omega_k' ,
\end{equation}
with $a_k, b_k$  constants, and $\omega_k, \omega_k'$ defined by
\[
\begin{split}
&
\omega_k=\frac{1}{\sqrt \lambda_k} f(r)^{1/2} w_k(r) d_S\eta_k+(-1)^{(N+1)/2} (f(r))^{-1/2} w_k(r) \eta_k\wedge dr,\\
& \omega_k'=*_S \, \omega_k,
\end{split}
\]
where $d_S$ and $*_S$ are the differential  and star operators on $S^N$, respectively; and $w_k(r)$ is defined as
\[
w_k(r)=\frac{(\tanh r/2)^{\sqrt\lambda_k-1/2}}{\cosh r/2}.
\]

Now assume that $p\leq 2N/(N+1)$.  Let $\omega$ be an $L^p$ harmonic $(N+1)/2$-form.  Recall from the proof of Theorem \ref{thm1} earlier in this section that the heat operator is bounded from $L^p$ to $L^q$ for all $q\geq p$ over the hyperbolic space.   Since the heat operator preserves harmonic forms, $e^{-t \Delta} \omega = \omega$, we have that  $\omega$ must also belong to $L^2$ and hence must be of the type ~\eqref{eta1}.

We identify $\mathbb H^{N+1}\backslash \{0\}$ to $\R^{N+1}\backslash \{0\}$ and then to $\R^+\times S^N$. Let $\langle\,,\,\rangle_0$ and $|\cdot|_0$ be the pointwise inner product with respect to the product metric of $\R^+\times S^N$.
 Assume that the coefficient $a_\ell$ in $\omega$ is not zero. We consider
\begin{equation} \label{6_e1}
\left|\int_{[1,\infty)\times S^N}\langle d_S\eta_\ell,\omega\rangle_0\,\frac{1}{1+r} dm\right|\leq C \int_{[1,\infty)\times S^N}|\omega|_0\,\frac{1}{1+r}\, dm,
\end{equation}
where $dm=dr\wedge dV_{S^N}$, and $ dV_{S^N}$ denotes the standard measure on the sphere of radius 1.
Using H\"older's inequality, we have
\begin{align*}
&\qquad
\int_{[1,\infty)\times S^N}|\omega|_0\,\frac{1}{1+r}\, dm\\
&\leq
\left(\int_{[1,\infty)\times S^N}|\omega|^p_0\, f^{-\frac{N+1}{2}p+N}\, dm\right)^{1/p}\cdot
\left(\int_{[1,\infty)\times S^N}f^{(\frac{N+1}{2}-\frac Np)q}\,\frac{1}{(1+r)^q}dm\right)^{1/q}\\
&\leq C\|\omega\|_{L^p(\mathbb H^{N+1})}\cdot \left(\int_{[1,\infty)\times S^N}\frac{1}{(1+r)^q}\,dm\right)^{1/q}<\infty,
\end{align*}
where $q=p/(p-1)$ is the conjugate of $p$.

However, the left side in \eqref{6_e1} is equal to
\[
a_\ell  \sqrt{2 \lambda_\ell} \int_{1}^\infty \left(\tanh\frac r2\right)^{\sqrt{\lambda_\ell}}\,\frac{1}{1+r} dr=\infty
\]
unless $a_\ell=0$ for all positive integers $\ell$. Similarly, $b_\ell=0$ for all positive integers $\ell$. This is a contradiction.

On the other hand, if $p>2N/(N+1)$ then for any fixed $k$, the form $\omega_k$ defined above is $L^p$-integrable by a straighforward computation. This completes the proof of the theorem.

\end{proof}

\appendix

\section{Stein Interpolation}

We state a version of the Stein Interpolation Theorem that we  use thoughout the paper. A more general version can be found in Davies~\cite{Davies}*{Section 1.1.6} and Stein and Weiss~\cite{StWe}. For a preliminary introduction we refer to the book of Reed and Simon~\cite{ReSi}.

Denote by $\|A\|_{p,q}$ the norm of a linear operator $A: L^{p} \rightarrow L^q $ and by $p^*$ the real number such that $1/p+ 1/p^*=1$.

\begin{lem}[The Stein Interpolation Theorem] \label{SInt} Let $1\leq p_0, p_1\leq \infty$ and $S=\{0\leq \textup{Re}\,z \leq 1\}$.
Suppose that  for all $z\in S$,  $T(z)$ is linear operator from $L^{p_0}(\Lambda^k(M))\cap L^{p_1}(\Lambda^k(M))$ to $L^{p_0}(\Lambda^k(M))+L^{p_1}(\Lambda^k(M)).$ Furthermore, assume that

\begin{enumerate}
\item $\<T(z) \omega, \eta\>$ is uniformly bounded and continuous on $S$ and analytic in the interior of $S$ whenever $\omega \in L^{p_0}(\Lambda^k(M))\cap L^{p_1}(\Lambda^k(M))$ and $ \eta \in L^{p_0^*}(\Lambda^k(M))\cap L^{p_1^*}(\Lambda^k(M))$.

\item For all $y\in \mathbb{R}$
\[
\|T(iy) \omega \|_{p_0} \leq  M_0 \, \|\omega \|_{p_0}
\]
for all $\omega \in L^{p_0}(\Lambda^k(M))\cap L^{p_1}(\Lambda^k(M))$.

\item  For all $y\in \mathbb{R}$
\[
\|T(1+iy) \omega \|_{p_1} \leq  M_1 \, \|\omega \|_{p_1}
\]
for all $\omega \in L^{p_0}(\Lambda^k(M))\cap L^{p_1}(\Lambda^k(M))$.
\end{enumerate}

Then for each $t\in (0,1)$  and $\omega \in L^{p_0}(\Lambda^k(M))\cap L^{p_1}(\Lambda^k(M))$
\[
\|T(t)\omega \|_{p_t} \leq M_0^{1-t}\, M_1^{t} \,\|\omega \|_{p_t}
\]
where $1/p_t = t/p_1 + (1-t)/p_0$.  Hence $T(t)$ can be extended to a bounded operator on $L^{p_t}(\Lambda^k(M))$ with norm at most $ M_0^{1-t}\, M_1^{t}$.
\end{lem}

\section{Some elementary results}\label{appB}
\begin{lem}\label{blem40}
Using the assumptions on functions $f$ and $g$ in \S~\ref{S4}, we have
\[
|g^{(j)}(w^2)|\leq\frac{C'_j}{(1+|w|)^{j}}
\]
 \end{lem}
 \begin{proof}
 We can prove, inductively, that
 \begin{equation}
 g^{(j)}(x)=\sum_{k=1}^j \sigma_{j,k}f^{(k)}(\sqrt{x}) x^{\frac 12 k-j}
 \end{equation}
hence
\begin{equation}
| g^{(j)}(w^2)|=|\sum_{k=1}^j \sigma_{j,k}f^{(k)}(w) w^{k-2j}|\leq C_j |w|^{-j}
 \end{equation}
 where $w$ is in the horizontal strip. Since $g^{(j)}(w^2)$ is holomorphic near $w=0$, we conclude the result of the lemma.
 \end{proof}

 \begin{lem}\label{blem41}
 Let $\sigma(w)$ be a bounded  holomorphic function defined on the strip
 \[
\{w\mid |{\rm Im}\,w|<\gamma_o/2+\eps_o\}\subset\C.
 \]
 Assume that $\sigma(w)\to 0$ as $w\to\infty$, and for any $| \tau |\leq (\gamma_o+\eps_o)/2$
 \[
 \int_\R |\sigma(x+i\tau )|dx\leq C
 \]
 for a constant $C$ independent of $\tau$.
 Then Fourier transform $\hat \sigma(\xi)$ of $\sigma(t)$ satisfies
\[
|\hat \sigma(\xi)|\leq Ce^{-(\frac{\gamma_o}{2}+\frac{\eps_o}{2}) \xi}
\]
for any real number $\xi$.
 \end{lem}

 \begin{proof} By definition,
 \[
\hat \sigma(\xi)=\int_{\R} \sigma(t)e^{-i\xi t} dt.
\]
Since $\sigma$ is holomorphic  and $\sigma(w)\to 0$ as $w\to\infty$,  by  Cauchy's  integral theorem over appropriate contours, we
can write
\[
\hat \sigma(\xi)=\int_{\R} \sigma (t-  i(\frac{\gamma_o} 2+\frac{\eps_o}{2}))e^{-i\xi (t- i(\frac{\gamma_o}2+\frac{\eps_o}{2}))} dt.
\]
Thus
\[
|\hat \sigma(\xi)|\leq e^{-(\frac{\gamma_o}{2}+\frac{\eps_o}{2}) \xi} \int_{\R}\left| \sigma (t-  i(\frac{\gamma_o} 2+\frac{\eps_o}{2}))\right|dt
\]
The lemma follows from the assumption on $\sigma(z)$.
 \end{proof}

 \begin{lem}[Taylor's formula]\label{blem42} Let $\alpha> 0$ and $j$ an integer.
For $j\geq 1$  define
 \[
 b_j=\frac{\alpha^{j}}{(j-1)!}\int_0^1 g^{(j)}(x+t\alpha)\, t^{j-1} dt.
 \]
 Then for any $N>1$, we have
 \[
 g(x)=\sum_{j=0}^{N-1}(-1)^{j}\frac{\alpha^j}{j!}\, g^{(j)}(x+\alpha)+(-1)^Nb_N.
 \]
 \end{lem}

 \begin{proof}
 If $j>1$, we have
 \[
 b_j=\frac{\alpha^{j-1}}{(j-1)!}\int_0^1 t^{j-1} d \left(g^{(j-1)}(x+t\alpha)\right)=\frac{\alpha^{j-1}}{(j-1)!}g^{(j-1)}(x+\alpha)-b_{j-1},
 \]
 and
 \[
 b_1=g(x+\alpha)-g(x).
 \]
 The lemma is proved.
 \end{proof}

Let $c\leq \lambda_1=\inf \sigma(2,k)$  and $\alpha>0$. Let $L=\sqrt{\Delta-c}$. Then by the above lemma  and functional calculus, we have
\begin{equation}\label{functioncalculus}
f(L)= g(\Delta -c)=A_N+B_N
\end{equation}
for any positive integer $N$,
  where
  \begin{align}\label{46}
  \begin{split}
  &A_N=\sum_{j=0}^{N-1} (-1)^j\frac{\alpha^j}{j!} g^{(j)}(\Delta  -c+\alpha);\\
  & B_N=\frac{(-1)^N\alpha^{N}}{(N-1)!}\int_0^1 g^{(N)}(\Delta  -c+t\alpha)t^{N-1} dt.
  \end{split}
    \end{align}


\end{document}